\documentclass[12pt,a4paper]{amsart}
\usepackage{amsfonts}
\usepackage{amsthm}
\usepackage{amsmath}
\usepackage{hyperref}
\usepackage{amscd}
\usepackage{amssymb}
\usepackage{enumitem}
\usepackage[latin2]{inputenc}
\usepackage{t1enc}
\usepackage[all]{xy}
\usepackage[mathscr]{eucal}
\usepackage{indentfirst}
\usepackage{graphicx}
\usepackage{graphics}
\usepackage{pict2e}
\usepackage{epic}
\usepackage{tikz-cd}
\DeclareMathOperator{\id}{id}
\usepackage{amsfonts}
\usepackage{mathtools}
\usepackage{tikz}
 \usetikzlibrary{matrix,arrows,decorations.pathmorphing}
\usepackage{hyperref}
\usepackage[mathscr]{eucal}
\usepackage{indentfirst}
\numberwithin{equation}{section}
\usepackage[margin=2.9cm]{geometry}
\usepackage{epstopdf}

\DeclareMathOperator{\HCal}{\mathcal{H}}

\DeclareMathOperator{\HomSh}{\mathscr{H}om}

\newcommand{\multGpSch}{{\mathbb{G}_m}}

\newcommand{\tens}{\otimes}
\newcommand{\iso}{\simeq}
\newcommand{\defeq}{\mathrel{\mathop:}=}
\newcommand {\inj}{\hookrightarrow}

\newcommand {\Hom}{{\rm Hom}}
\newcommand {\Ext}{{\rm Ext}}
\newcommand{\dR}{\rm dR}
\newcommand{\cris}{\rm cris}
\newcommand{\Hdg}{\rm Hdg}
\newcommand{\totale}{\rm total}
\newcommand{\slope}{\rm slope}

\DeclareMathOperator{\Spec}{Spec}

\DeclareMathOperator{\Frob}{Frob}

\newcommand {\supp}[1]{{\rm Supp}(#1)}

\newcommand {\inv}{^{-1}}
\newcommand {\integ}{\mathbb Z}

\newcommand {\C}{\mathbb C}

\newcommand {\rat}{\mathbb Q}
\DeclareMathOperator{\DD}{D}
\DeclareMathOperator{\bb}{b}
\DeclareMathOperator{\coh}{coh}
\newcommand{\Dbcoh}{\DD^{\bb}_{\coh}}
\DeclareMathOperator{\perf}{perf}

\newcommand{\DperfAbs}{\DD_{\perf}}

\DeclareMathOperator{\Ker}{Ker}

\newcommand{\OCal}{\mathcal{O}}

\DeclareMathOperator{\can}{can}

\DeclareMathOperator{\ACal}{\mathcal{A}}

\DeclareMathOperator{\Aut}{Aut}

\newcommand{\surj}{\twoheadrightarrow}
\theoremstyle{plain}
\newtheorem{Th}{Theorem}[section]
\newtheorem{Lemma}[Th]{Lemma}
\newtheorem{Cor}[Th]{Corollary}
\newtheorem{Prop}[Th]{Proposition}
\newtheorem{Prop/Def}[Th]{Proposition/Definition}

 \theoremstyle{definition}
\newtheorem{Def}[Th]{Definition}
\newtheorem{Conj}[Th]{Conjecture}
\newtheorem{Question}[Th]{Question}
\newtheorem{Rem}[Th]{Remark}
\newtheorem{Ex}[Th]{Example}
\newcommand{\JCal}{\mathcal{J}}
\newcommand{\ZCal}{\mathcal{Z}}

\DeclareMathOperator{\Alb}{Alb}

\usepackage{amscd}
\usepackage[latin2]{inputenc}
\usepackage{t1enc}

\usepackage{mathtools}
\usepackage{tikz}
 \usetikzlibrary{matrix,arrows,decorations.pathmorphing}
\usepackage{hyperref}
\usepackage[mathscr]{eucal}
\usepackage{indentfirst}
\DeclareMathOperator{\et}{\acute et}
\DeclareMathAlphabet{\mathantt}{OT1}{antt}{li}{it}

\usepackage{graphicx}
\usepackage{graphics}
\usepackage{pict2e}

\DeclareMathOperator{\character}{char}
\DeclareMathOperator{\Frac}{Frac}
\usepackage{epic}
\DeclareMathOperator{\Pic}{Pic}
\newcommand{\Dcal}{\mathcal{D}}

\usepackage[all]{xy}
\newcommand{\LCal}{\mathcal{L}}

\numberwithin{equation}{section}
\usepackage[margin=2.9cm]{geometry}
\usepackage{epstopdf}

\theoremstyle{plain}

 \theoremstyle{definition}

\DeclareMathOperator{\red}{red}
\newcommand{\FCal}{\mathcal{F}}
\newcommand{\XCal}{\mathcal{X}}
\newcommand{\YCal}{\mathcal{Y}}

\newcommand{\ECal}{\mathcal{E}}

\DeclareMathOperator{\CH}{CH}
\DeclareMathOperator{\FMP}{FMP}

\usepackage{color}  
\usepackage{hyperref}
\hypersetup{
    colorlinks=false, 
    linktoc=all,    
    linkcolor=blue, 
}

\begin{document}

\title[Fourier-Mukai Partners of Abelian Varieties and K3 surfaces]{Fourier-Mukai Partners of Abelian Varieties and K3 surfaces in Positive and Mixed Characteristics}

\author[R. Kurama]{Riku Kurama}

\address{Department of Mathematics, University of Michigan, Ann Arbor, MI 48109} 

\email{rkurama@umich.edu}
\maketitle

\begin{abstract}
We study Fourier-Mukai equivalences of (families of) abelian varieties and K3 surfaces in positive and mixed characteristics. We first prove 
in any characteristics 
that Fourier-Mukai partners of abelian varieties are again abelian varieties. We subsequently focus on the canonical lifts of ordinary abelian varieties and ordinary K3 surfaces. For such schemes, we show that Fourier-Mukai equivalences on the special fibers can be lifted to the canonical lifts. We also prove that the relative Fourier-Mukai partners of the canonical lifts are in bijection with the Fourier-Mukai partners of the special fibers. We conclude by demonstrating that the last result can be used to recover the ordinary case of a result originally proved by Honigs, Lombardi and Tirabassi.
\end{abstract}
\tableofcontents

\begingroup
\let\thefootnote\empty
\footnotetext{Keywords: derived equivalence, abelian varieties, K3 surfaces, positive and mixed characteristics.}
\footnotetext{MSC2020: 14F08, 14G17, 14G45, 14J28, 14K05.}
\endgroup

\section{Introduction}
One of the first attempts to understand the landscape of Fourier-Mukai equivalences in positive characteristics is due to Lieblich and Olsson (\cite{LO}), where a key technique was the lifting of Fourier-Mukai transforms (or their kernels) from characteristic $p$ to characteristic $0$.
Motivated by this, a natural question is whether there are interesting classes of examples where one can lift Fourier-Mukai equivalences from characteristic $p$ to characteristic $0$. 
In \cite{DelCris}, starting with an ordinary abelian variety or an ordinary K3 surface (the precise assumptions are as in Theorem \ref{MainLift} below), Deligne exhibits the so-called ``canonical coordinates'' on its mixed-characteristic formal moduli space and uses them to constructs its (algebraizable) canonical lift over the ring of Witt vectors (see also \cite{DelK3}, \cite{KatzST}, \cite{BranTael}).
The first contribution of this paper is that such canonical lifts provide examples where kernels of Fourier-Mukai equivalences can be always lifted, answering the question above: 

\begin{Th}[Lifting Fourier-Mukai equivalences; Propositions \ref{CanLiftAVKer}, \ref{K3LiftKer}]\label{MainLift}
Let $k$ be an algebraically closed field of characteristic $p>2$.
Let $X$ and $Y$ be ordinary abelian varieties (resp.\ ordinary K3 surfaces), and assume that we have a kernel $E\in\Dbcoh(X\times_k Y)$ defining a Fourier-Mukai equivalence.
Then $E$ lifts to a perfect complex $\ECal\in\DperfAbs(X_{\can}\times_{W(k)}Y_{\can})$ defining a relative Fourier-Mukai equivalence over $W(k)$, where $X_{\can}$ and $Y_{\can}$ denote the (algebraized) canonical lifts of $X$ and $Y$ over $W(k)$.
\end{Th}
\begin{Rem}
    Note that we are not claiming here that there is a canonical choice for the lift of the kernel (see Remark \ref{NonUniqueLift}). However, in the K3 surface case, the lift is in fact unique (Proposition \ref{K3LiftKer}).  
As the proof will make clear, the K3 surface case of this theorem is a consequence of a theorem from \cite{LO}.
\end{Rem}
Applying this theorem to the Chern character of the kernel, we obtain the following special case of the $p$-adic variational Hodge conjecture (\cite{Emerton}):
\begin{Cor}\label{MainLiftCor}
Let $X$, $Y$ and $E$ be as in Theorem \ref{MainLift}. We consider the Chern classes $c_i(E)\in\CH^{i}(X\times_kY)_\rat$ in the rational Chow groups. If we view $X\times_kY$ as the special fiber of the $W(k)$-scheme $X_{\can}\times_{W(k)}Y_{\can}$, the $p$-adic variational Hodge conjecture holds for the class $c_i(E)$.
\end{Cor}
\begin{Rem}
    Again, the K3 surface case of this claim is a corollary of a result in \cite{LO}. In the abelian variety case, we in fact prove a slightly stronger statement (Proposition \ref{padicVarHdgAV}).
\end{Rem}

Moreover, we describe the relative Fourier-Mukai partners of such canonical lifts as follows (the case of ordinary K3 surfaces was essentially already known by \cite{TanK3}):
\begin{Th}[Relative Fourier-Mukai partners of canonical lifts; Propositions \ref{FMPofCanLiftAV}, \ref{FMPofCanLiftK3}]\label{MainFMPofCanLift}
 Let $k$ be an algebraically closed field of characteristic $p>2$.
Let $X$ be an ordinary abelian variety (resp.\ an ordinary K3 surface). Then,
the restriction to the special fiber defines a bijection
\[
\FMP(X_{\can}/W(k))\longrightarrow \FMP(X/k),
\] where $\FMP(X_{\can}/W(k))$ denotes the set of (isomorphism classes of) $W(k)$-relative Fourier-Mukai partners of $X_{\can}$, and $\FMP(X/k)$ is similarly defined.

In particular, $\FMP(X_{\can}/W(k))$ is finite. 
\end{Th}
As a sample application of this theorem, we reprove the following special case of a result from \cite{Honigs}:
\begin{Cor}[Special case of {\cite[Theorem 1.1]{Honigs}; Proposition \ref{HonigsApp}}]\label{MainHonigsApp}
  Let $S$ be a hyperelliptic surface over an algebraically closed field $k$ of characteristic $p>3$, and let $A$ be its canonical cover. Additionally assume that $A$ is ordinary.
    Then any Fourier-Mukai partner of $A$ is isomorphic to $A$ or the dual abelian variety $\widehat{A}$.
\end{Cor}
Here, a \textbf{hyperelliptic surface} means a smooth projective minimal surface with the properties that the canonical divisor is numerically trivial, $b_2=2$ and the fibers of the Albanese map are (smooth) elliptic curves. Since the canonical bundle of such a surface $S$ has a finite order, by a standard construction we can produce an \'etale cyclic cover $A$ of $S$ which trivializes the canonical bundle, which is called the \textbf{canonical cover of $S$} (\cite[\S7.3]{HuyFMT}). It is known that the canonical cover of a hyperelliptic surface is an abelian surface (see
\cite[\S4]{Honigs} for details).

The proof of Theorem \ref{MainFMPofCanLift} for the case of abelian varieties uses the following fact which is interesting in its own right:
\begin{Th}[Fourier-Mukai partners of abelian varieties; Proposition \ref{FMPAV}]\label{FMPofAVisAV}
   Let $k$ be an algebraically closed field of arbitrary characteristic, and let $X$ be an abelian variety over $k$. Then, any (smooth projective) Fourier-Mukai partner of $X$ is again an abelian variety.
\end{Th}
This result was previously known for algebraically closed fields of characteristic 0 (\cite{Huy}, \cite{rosay2009some}) but our proof in \S\ref{secFMPofAV} appears new in positive characteristics\footnote{The author was later made aware of the work \cite{rosay2009some} which follows a similar general proof strategy
but uses the characteristic assumption 
by an anonymous referee. The work \cite{Huy} takes a different approach and uses the Beauville-Bogomolov decomposition.}
We will focus on ordinary abelian varieties and ordinary K3 surfaces in this paper, but there are many other varieties which are now known to admit the theory of canonical lifts. 
For instance, it has been known for some time that (algebraizable) canonical lifts exist for ordinary varieties with trivial tangent bundles
(\cite{MS}). More recently, Brantner and Taelman constructed (algebraizable) canonical lifts for a more general class of Calabi-Yau varieties, so one could ask if the main results from this paper generalize to these varieties. 
Here is the precise statement about the existence of canonical lifts proved by Brantner and Taelman:
\begin{Th}[Brantner Taelman, {\cite[Theorems B,C]{BranTael}}]\label{GeneralCanLift}
Let $k$ be a perfect field of characteristic $p>0.$ Let $X$ be a smooth projective geometrically irreducible Calabi-Yau $k$-variety.
We additionally assume
that $X$ is Bloch-Kato 2-ordinary and that $H^{\dim X}_{\et}(X_{\overline{k}},\integ_p)$ is torsion-free.
Then, $X$ has a canonical lift $X_{\can}$ over $\Spec W(k)$.
\end{Th}
Motivated by Theorems \ref{MainLift} and \ref{MainFMPofCanLift}, we propose:
\begin{Question}
Let $X,Y$ be two varieties over a perfect field $k$ of characteristic $p>0$, satisfying the assumptions of Theorem \ref{GeneralCanLift}.
If a kernel $E\in\Dbcoh(X\times_kY)$ defines a derived equivalence, 
does $E$ always lift to a perfect complex $\ECal\in\DperfAbs(X_{\can}\times_{W(k)}Y_{\can})$, thereby defining a $W(k)$-relative Fourier-Mukai equivalence between $X_{\can}$ and $Y_{\can}$?
\end{Question}
\begin{Question}\label{MainQuestion2}
Let $X$ be a variety over a perfect field $k$ of characteristic $p>0$, satisfying the assumptions of Theorem \ref{GeneralCanLift}.
Then does 
the restriction map to the special fiber
\[
\FMP(X_{\can}/W(k))\longrightarrow \FMP(X/k)
\] 
 define a bijection?

\end{Question}
\begin{Rem}
    Assuming for a moment that $k$ is algebraically closed, if the answer to Question \ref{MainQuestion2} is yes, it will imply that $\FMP(X_{\can}/W(k))$ is a countable set, as this is known for $\FMP(X/k)$ (\cite[\href{https://stacks.math.columbia.edu/tag/0G11}{Tag 0G11}]{stacks-project}, \cite{Toen}). After completing the draft for this paper, the author found a direct proof of this countability result (\cite{Kurama})\footnote{We are considering only smooth projective schemes over the ring of Witt vectors here, but the main result of \cite{Kurama} is more general and concerns smooth proper schemes over a noetherian base.}, so it provides some positive evidence toward the question.
\end{Rem}
After completing the draft for this paper, the author was notified of Rienks's recent paper on arXiv (\cite{Rien}) which can be used to prove Theorem \ref{MainLift} under a certain characteristic assumption. (Namely, \cite[Corollary 1.6]{Rien} can be applied given that the characteristic $p>2$ is greater than twice the dimension of the variety).
One feature of the results in \cite{Rien} is that they also apply to Calabi-Yau varieites other than abelian varieties and K3 surfaces, and
the author hopes to address the above two questions in future works using similar techniques.

\subsection*{Organization of the paper}
\S\ref{secRFMT} recalls some basic facts about Fourier-Mukai transforms relative to a base.
\S\ref{secFMPofAV} is devoted to the proof of Theorem \ref{FMPofAVisAV}.
After this, the paper will address Theorem \ref{MainLift} 
(cf.\ Proposition \ref{CanLiftAVKer}, Proposition \ref{K3LiftKer}),
Theorem \ref{MainFMPofCanLift} 
(cf.\ Proposition \ref{FMPofCanLiftAV}, Proposition \ref{FMPofCanLiftK3}) 
and Corollary \ref{MainLiftCor}
(cf.\ Proposition \ref{padicVarHdgAV}). We will first explain the abelian variety case in details (\S\ref{secAV}) and then 
show that an analogous argument works for K3 surfaces rather briefly (\S\ref{secK3}).
Section \ref{secHonigs} will prove Corollary \ref{MainHonigsApp} as an application of Theorem \ref{MainLift}. The appendix \S \ref{secChern} serves to clarify the compatibility between the de Rham and crystalline Chern classes. 
\subsection*{Conventions}
A variety over a field $L$ is a finite type $L$-scheme. 

For a perfect field $k$ of positive characteristic, $W(k)$ denotes its ring of Witt vectors. When $k$ is clear from context, we also write $W\defeq W(k), K\defeq \Frac(W(k))$.
When $k$ is a perfect field of positive characteristic and $X$ is a smooth projective $k$-scheme, its $i$-th crystalline cohomology is denoted by
$H_{\cris}^i(X/W)$. We also define its rational version $H_{\cris}^i(X/K)\defeq H_{\cris}^i(X/W)\tens_{W}K$.

For a $W(k)$-scheme  $\XCal$, its special fiber is denoted by $\XCal_s$, and its generic fiber (resp.\ its geometric generic fiber) is denoted by $\XCal_\eta$ (resp.\ $\XCal_{\overline{\eta}}$).

For a noetherian scheme $X$, $\Dbcoh(X)$ and $\DperfAbs(X)$ denote the bounded derived category of coherent sheaves on $X$ and the category of perfect complexes on $X$. 

For a morphism $X\longrightarrow S$ of schemes, its $i$-th relative de Rham cohomology sheaf is denoted by $\HCal_{\dR}^i(X/S)(=\HCal^i(Rf_*\Omega_{X/S}^\bullet))$. If $S$ is affine, we also write it as $H_{\dR}^i(X/S).$
\subsection*{Acknowledgement}
The author would like to thank his advisor Alexander Perry for continual support, helpful  discussions and suggestions.  
The author thanks Sean Cotner, Hyunsuk Kim, Samuel Moore, Martin Olsson, Saket Shah and Gleb Terentiuk for helpful discussions and interest. The author also thanks Sean Cotner for suggesting a slight simplification in the proof of Proposition \ref{FMPAV}. The author thanks reviewers  for helpful comments.

During the preparation of this paper, the author was partially supported by NSF grant DMS-2052750.
\section{Relative Fourier-Mukai transform}\label{secRFMT}
We fix a noetherian base scheme $S$ throughout this section and review the relative Fourier-Mukai theory over $S$ . The results from this section are more general than what is needed for the remainder of the paper (in fact we will only need the case of smooth projective schemes over the base $S=\Spec W(k)$ or $S=\Spec k$).

\begin{Def}[Relative Fourier-Mukai transform]
Given two smooth proper $S$-schemes $\XCal,\YCal$ and 
$\ECal\in \DperfAbs
(\XCal\times_S\YCal)$, we define the associated \textbf{relative Fourier-Mukai transform} $\Phi_\ECal:\Dbcoh(\XCal)\longrightarrow\Dbcoh(\YCal)$ by the formula:
\[
\Phi_\ECal(-)=Rq_*(Lp^*(-)\tens^L \ECal)
\] where the maps $p,q$ are the projections $p:\XCal\times_S\YCal\longrightarrow \XCal,q:\XCal\times_S\YCal\longrightarrow\YCal$. The map $q$ is a proper morphism of noetherian schemes, so $Rq_*$ respects the bounded derived categories (even if the relevant schemes have infinite Krull dimensions), and $\Phi_\ECal$ is well-defined. 
The object $\ECal$ in this context is called the \textbf{kernel} of the relative Fourier-Mukai transform.
\end{Def}
\begin{Rem}
Note that by noetherianity, any perfect complex on $\XCal\times_S\YCal$ is bounded and hence lies in $\Dbcoh(\XCal\times_S\YCal)$ (\cite[\href{https://stacks.math.columbia.edu/tag/08E8}{Tag 08E8}]{stacks-project},
\cite[\href{https://stacks.math.columbia.edu/tag/08JL}{Tag 08JL}]{stacks-project}).
In a more general setting, it's natural to work with $\XCal$-relatively perfect complexes to define Fourier-Mukai transforms (\cite[Proposition 2.7]{RMS}). However, perfect complexes suffice in our setting due to smoothness (\cite[Proposition 3.12]{TLS}).
\end{Rem}
Let $\XCal,\YCal,\ZCal$ be smooth proper $S$-schemes. Given perfect complexes on two fiber products $\ECal\in\DperfAbs(\XCal\times_S\YCal),\FCal\in\DperfAbs(\YCal\times_S\ZCal)$, recall that we have the composition $\FCal\circ\ECal\defeq Rp_{1,3*}(Lp_{1,2}^*\ECal\tens^LLp_{2,3}^*\FCal)\in\DperfAbs(\XCal\times_S\ZCal)$ where $p_{i,j}$ is the projection to the $(i,j)$ components of the product $\XCal\times_S\YCal\times_S\ZCal$ (the composition is a perfect complex by \cite[\href{https://stacks.math.columbia.edu/tag/0DJT}{Tag 0DJT}]{stacks-project} for instance). As the name suggests, the composition $\FCal\circ\ECal$ has the property that $\Phi_{\FCal\circ \ECal}=\Phi_{\FCal}\circ\Phi_{\ECal}.$
\begin{Prop/Def}\label{RelEquiv}
Let $\XCal,\YCal$ be two smooth proper $S$-schemes and choose an object
$\ECal\in \DperfAbs
(\XCal\times_S\YCal)$. The following two conditions are equivalent:
\begin{enumerate}
    \item For each point $s\in S$, the object 
    \[\ECal_s\defeq \ECal|_{(\XCal\times_S\YCal)\times_S\{s\}}\in \DperfAbs(\XCal_s\times_{\kappa(s)}\YCal_s)\] defines an equivalence of categories $\Phi_{\ECal_s}:\Dbcoh(\XCal_s)\longrightarrow\Dbcoh(\YCal_s)$.
    \item There exists another kernel $\FCal\in\DperfAbs(\YCal\times_S\XCal)$ such that 
    \[\ECal\circ \FCal\iso\OCal_{\Delta_\YCal},\quad 
    \FCal\circ \ECal\iso\OCal_{\Delta_{\XCal}}.\]
\end{enumerate}
A functor of form $\Phi_\ECal:\Dbcoh(\XCal)\longrightarrow \Dbcoh(\YCal)$ with $\ECal$ satisfying these equivalent conditions is called an \textbf{$S$-relative Fourier-Mukai equivalence.} 
Two smooth projective $S$-schemes $\XCal,\YCal$ which admit such a kernel are called \textbf{$S$-relative Fourier-Mukai partners}. Given a smooth projective $S$-scheme $\XCal$, we write $\FMP(\XCal/S)$ to denote the set of isomorphism classes of its $S$-relative Fourier-Mukai partners.
\end{Prop/Def}
\begin{proof}
 If we assume (2), the same statement holds after base changing along any $\{s\}\longrightarrow S$. In particular, the functor $\Phi_{\FCal_s}:\Dbcoh(\YCal_s)\longrightarrow\Dbcoh(\XCal_s)$ defines a quasi-inverse to $\Phi_{\ECal_s}$. This implies (1).

 Conversely, assume (1). We define kernels 
 \[\ECal_L\defeq R\HomSh_{\XCal\times_S \YCal}(\ECal,p^\times\OCal_\XCal),\quad
 \ECal_R\defeq R\HomSh_{\XCal\times_S \YCal}(\ECal,q^\times\OCal_\YCal)\] where 
$p^\times$ denotes the right adjoint to $Rp_*$ (and similarly for $q^\times$) as in \cite[\S3]{Neeman}.
Since the structure morphisms are assumed smooth proper, $p^\times\OCal_\XCal$ is locally a line bundle with an appropriate cohomological shift (\cite[Theorem 3.2.1]{Neeman}). Formal computations using Grothendieck duality, $p^\times(-)\iso Lp^*(-)\tens p^\times\OCal_X$ (see \cite[Remark 3.1.5]{Neeman}) and $\id\iso R\HomSh(-,R\HomSh(-,p^\times\OCal_X),p^\times\OCal_X)$ (and the similar results for $q$) show that $\Phi_{\ECal_L}$ and $\Phi_{\ECal_R}$ define the left and right adjoints of $\Phi_\ECal$.
 Also, the unit and counit of the adjunctions exist at the level of kernels (\cite[\S3]{LO}):
 \[
 \OCal_{\Delta_\YCal}
 \longrightarrow \ECal\circ\ECal_R, \quad  \ECal_L\circ \ECal\longrightarrow \OCal_{\Delta_\XCal}.
 \]
 Let $C_1,C_2$ be the cones of these two morphisms. 
 As in \cite[\S3]{LO}, when the base $S$ is a field, $\Phi_\ECal$ defines an equivalence if and only if these two maps are quasi-isomorphisms. Since the formation of these two maps are compatible with the base change of $S$, (1) implies that $C_1$ and $C_2$ vanish for all fibers above $S$. Then, the support of $C_1$ and $C_2$ are empty, i.e. they are trivial. Thus, the unit and counit define quasi-isomorphisms. We then see
 \[
 \ECal_L=\ECal_L\circ\OCal_{\Delta_\YCal}\iso \ECal_L\circ \ECal\circ \ECal_R\iso \OCal_{\Delta_\XCal} \circ\ECal_R=\ECal_R,
 \] so that we obtain (2).
\end{proof}
\begin{Rem}
While the above definition applies to any smooth proper schemes over $S$, we will only consider Fourier-Mukai partners that are  
smooth projective over $S$ in this paper. 
\end{Rem}
\begin{Rem}\label{baseChangeFMP}
Fourier-Mukai equivalence is stable under base change. More precisely, given $\ECal\in\DperfAbs(\XCal\times_S\YCal)$ defining a $S$-relative Fourier-Mukai equivalence, for any morphism $T\longrightarrow S$ of noetherian schemes, the pullback of $\ECal$ to $\XCal_T\times_T\YCal_T$ defines a $T$-relative Fouirier Mukai equivalence between $\XCal_T$ and $\YCal_T$.
\end{Rem}
\begin{Lemma}[Fourier-Mukai equivalence is an open condition, see {\cite[Lemma 2.7]{Olsson-Rouquier}}]\label{liftEquiv} Let $\XCal,\YCal$ be smooth projective over $S$, and choose $\ECal\in\DperfAbs(\XCal\times_S\YCal)$.
If the restriction of $\ECal$ to the fiber over $s\in S$ defines an equivalence, there is an open neighborhood $U\subset S$ of $s$ such that the restriction of $\ECal$ to $\XCal_U\times_U\YCal_U$ defines a $U$-relative Fourier-Mukai equivalence.
\end{Lemma}
\begin{proof}
As in Proposition/Definition \ref{RelEquiv}, 
we look at the exact triangles
 \[
 \OCal_{\Delta_\YCal}
 \longrightarrow \ECal\circ\ECal_R\longrightarrow C_1\longrightarrow , \quad  \ECal_L\circ \ECal\longrightarrow \OCal_{\Delta_\XCal}\longrightarrow C_2\longrightarrow .
 \]
 Again, by \cite[\S3]{LO}, the restrictions to fibers $(C_1)_s,(C_2)_s$ vanish due to our hypothesis. Then, $\supp{C_1},\supp{C_2}$
 avoid $\pi\inv(s)$ where $\pi:\XCal\times_S\YCal\longrightarrow S$ is the structure map. 
 These supports are closed, so that by properness, $\pi(\supp{C_1}\cup\supp{C_2})\subset S$ is a closed subset not containing the point $s$. We can set our open neighborhood to be $U=S\setminus \pi(\supp{C_1}\cup\supp{C_2})$.
\end{proof}
\begin{Cor}\label{LiftEquiv}
    Let $\XCal,\YCal$ be smooth proper $S$-schemes, where $S$ is the spectrum of a noetherian local ring with the unique closed point $s\in S$. If the restriction of a perfect complex $\ECal\in\DperfAbs(\XCal\times_S\YCal)$ to the fiber above $s$ defines a Fourier-Mukai equivalence, $\ECal$ defines a Fourier-Mukai equivalence relative to $S$.
\end{Cor}
\section{Fourier-Mukai partners of abelian varieties}\label{secFMPofAV}
Let $k$ be an algebraically closed field of arbitrary characteristic. All varieties in this section are defined over $k$.
\begin{Prop}[Combination of {\cite[Th\'eor\'eme 4.18]{Rou} and \cite[Appendix]{Honigs-albanese}}]\label{RouPopaSch}
    Let $X$ and $Y$ be derived equivalent smooth projective connected varieties over $k$.
    Then, we have the following results:
    \begin{enumerate}
        \item We have an isomorphism of group schemes $\Pic^0_{X/k}\rtimes\Aut^0_{X/k}\iso \Pic^0_{Y/k}\rtimes\Aut^0_{Y/k}$. Here, $\Aut_{X/k}^0$ is the identity component of the group scheme of automorphisms of $X$ over $k$ (and similarly for $Y$).
        \item We have
$\dim\Pic^0_{X/k}=\dim\Pic^0_{Y/k},\dim\Aut^0_{X/k}=\dim\Aut^0_{Y/k}.$
    \end{enumerate}
\end{Prop}
\begin{proof}
   The first statement is due to \cite{Rou} (also see \cite{rosay2009some}).
   The second part follows from the appendix of \cite{Honigs-albanese} which extends a result of \cite{PS} from characteristic 0.
\end{proof}
The following result was known for $k=\C$ (\cite[Theorem 3.2]{rosay2009some},\cite[Proposition 3.1]{Huy},\cite{PS}) but appears new in positive characteristics:
\begin{Prop}\label{FMPAV}
    Let $A$ be an abelian variety. If $X$ is another smooth projective variety derived equivalent to $A$, then $X$ must be an abelian variety.
\end{Prop}
\begin{proof}
    Let $g=\dim A=\dim X$, where we note that the dimension is derived-invariant. The first half of the proposition above shows $\Pic^0_{X/k}\rtimes \Aut_{X/k}^0\iso A\times {A}^{\vee}$, from which we see that 
    $\Pic^0_{X/k}$, $\Aut^0_{X/k}$ are reduced. As in \cite[footnote on pp.4]{PS}, the semi-direct product on the left hand side is a product, 
    and we deduce that 
    $G\defeq\Aut_{X/k}^0$ is an abelian variety. 
     The second half of the proposition above says that $\Aut_{X/k}^0$ has dimension $g$.
    
    Choose a rational point $x\in X(k)$ and consider the action map
    \[f:G\longrightarrow X;\quad g\mapsto g\cdot x.\]
    We consider the stabilizer $S\inj G$ which is the fiber product $S\defeq \{x\}\times_XG$ (closed immersion is a monomorphism, so $h_S\inj h_G$ defines a subfunctor, and it's easy to see that it represents the subgroup functor of stabilizers).
    
    Note that $f$ acts on $T$-points as follows. Given a $T$-linear automorphism 
    \[(\sigma: X_T\iso X_T)\in G(T),\]
    its image $f(\sigma)\in X(T)$ is the composition 
    \[
    \sigma\circ x_T:T\longrightarrow X_T\longrightarrow X_T
    \] viewed as an element in $\Hom_T(T,X_T)=\Hom_k(T,X)$.

    Thus, the inclusion $S^0\inj G$ corresponds to some $S^0$-linear isomorphism $\tau:X\times S^0\longrightarrow X\times S^0 $ such that the composition \[ 
    \pi\circ \tau\circ \iota:
    \{x\}\times S^0\inj X\times S^0\longrightarrow X\times S^0 \longrightarrow X\] is the constant map to $x$, where $\pi$ is the first projection, and $\iota$ is the obvious inclusion map. By the rigidity lemma (\cite[Proposition 6.1]{GIT}), the map 
    \[\pi\circ \tau:
     X\times S^0\longrightarrow X\times S^0 \longrightarrow X\]
     factors through the projection to $X$. On the other hand, the restriction of the map $\pi\circ \tau$ to $X\times \{e\}$ is the identity, so $\tau=\id_{X\times S^0}$. In view of the functor that $G$ represents, the inclusion $S^0\inj G$ factors through the unit section $e_G.$
     This shows $S^0=e_G$, so that $S$ is finite \'etale over $k$. By translation, the fiber of $f$ above any $k$-point on $X$ is finite. As $f$ is projective, it is then a finite map. By dimension reasons, it must be surjective. 
     
    Since the stabilizer $S$ is finite \'etale over $k$, the action map $G\longrightarrow X$ is also finite \'etale. 
     We claim that $S(k)$ is a singleton. 
     Choose $g\in S(k)$. We know that $f$ is surjective on $k$-points, so for any rational point $y\in X(k)$, there is some $h_y\in G(k)$ with $h_y\cdot x=y.$
     Since $g$ fixes the $k$-rational point $x$, we see \[g\cdot y=g h_y\cdot x=h_y g\cdot x= h_y\cdot x=y,
     \] where we used that $G(k)$ is abelian. This shows that $g\in G(k)=\Aut_{X/k}^\circ(k)$ corresponds to a $k$-automorphism $X\longrightarrow X$ which fixes every $k$-point. It follows that $g$ is trivial. Thus, $S(k)$ is the trivial group, and $S$ is the trivial group scheme over $k$. This shows that the action map $G\longrightarrow X$ is in fact an isomorphism.
\end{proof}
\begin{Rem}
The last paragraph of the proof can be replaced by an argument using the following fact (whose proof in arbitrary characteristics was communicated to the author by S.\ Cotner): If $f:X\longrightarrow Y$ is a finite \'etale morphism of smooth projective connected varieties, the induced map $\Alb(X)\longrightarrow \Alb(Y)$ is surjective on tangent spaces. However, our current proof appears simpler in arbitrary characteristics, given the work involved in justifying this fact. The alternative proof goes as follows.

The action map $G\longrightarrow X$ is finite \'etale, so the induced map of Albanese varieties $\Alb(G)=G\longrightarrow X\longrightarrow\Alb(X)$ is surjective on tangent spaces. Since $\dim\Alb(X)=\dim(((\Pic^0_{X/k})_{\red})^\vee)=\dim\Pic^0_{X/k}=g$, the three varieties $X,G,\Alb(X)$ have the same dimension. It follows that the map $X\longrightarrow\Alb(X)$ is \'etale and projective, hence surjective and finite as well. Thus, $X$ is an abelian variety.
\end{Rem}
\begin{Cor}\label{FMPASch}
Assume $\character(k)>0.$
    Let $\XCal,\ACal$ be relatively Fourier-Mukai equivalent smooth projective $W(k)$-schemes. If $\ACal$ is an abelian scheme, so is $\XCal$.
\end{Cor}
\begin{proof}
    Proposition \ref{FMPAV} implies that the special fiber $\XCal_s$ is an abelian variety. By standard deformation theory, we see that the formal completion $\widehat{\XCal}$ has an abelian formal scheme structure (cf.\ \cite[\S6.3]{GIT}).
    By \cite[Th\'eor\`eme 5.4.1]{EGA3}, the abelian formal scheme structure algebraizes, 
     so that $\XCal$ is an abelian scheme.
\end{proof}
\section{Abelian variety case}\label{secAV}
\subsection{Relative de Rham cohomology of abelian schemes}
Let $S$ be a locally noetherian scheme, and let $(F^\bullet,M)$ be a locally free coherent $\OCal_S$-module with a decreasing filtration (with non-negative indices). 
Then, the exterior power $\bigwedge^m M$ has the following filtration. 
\[
F^k\bigwedge^mM\defeq\sum_{i_1+i_2+...+i_m=k}F^{i_1}M\wedge F^{i_2}M\wedge...\wedge F^{i_s}M \subset \bigwedge^mM.
\]
Let $\ACal$ be an abelian scheme over a locally noetherian scheme $S$. Let $f:\ACal\longrightarrow S$ be the structure morphism.
\begin{Lemma}[{\cite[Proposition 2.5.2]{BBM}}]\label{BBMLemma}
\leavevmode
    \begin{enumerate}
        \item For each non-negative integer $m$, cup product defines an isomorphism of coherent $\OCal_S$-modules
\[
\cup_m:\bigwedge^m\HCal_{\dR}^1(\ACal/S)\longrightarrow \HCal_{\dR}^m(\ACal/S).
\] 
\item Hodge to de Rham spectral sequence for $f:\ACal\longrightarrow S$ degenerates at the $E_1$-page. 
\item The coherent sheaves $\HCal^m_{\dR}(\ACal/S),R^if_*(\Omega^j_{\ACal/S})$ are locally free and their formation is compatible with (underived) base change.
    \end{enumerate}
\end{Lemma}
Each de Rham cohomology sheaf $\HCal_{\dR}^i(\ACal/S)$ is equipped with the Hodge filtration, obtained from the naive truncation of the de Rham complex. As we will show next, this filtration has a rather simple description.
If we equip $\bigwedge^m\HCal_{\dR}^1(\ACal/S)$ with the exterior power filtration as above, we see from
\cite[\href{https://stacks.math.columbia.edu/tag/0FM7}{Tag 0FM7}]{stacks-project} that the map $\cup_m$ is a morphism of filtered finite locally free $\OCal_S$-modules. We claim:
\begin{Prop}\label{HdgFilWedg}
For each non-negative integer $m$, cup product defines an isomorphism of filtered $\OCal_S$-modules
\[
\cup_m:\bigwedge^m\HCal_{\dR}^1(\ACal/S)\longrightarrow \HCal_{\dR}^m(\ACal/S).
\] 
\end{Prop}
\begin{proof}
It suffices to show that this induces a surjective map on each filtered piece.
By Lemma \ref{HdgFilField} below, this is known if $S$ is the spectrum of a field. 
The general case can be reduced to the field case by the base change of Hodge filtration (Lemma \ref{HdgFilBaseChange} below) and Nakayama's lemma.
\end{proof}

Here are the lemmas used in the proof of the proposition:
\begin{Lemma}\label{HdgFilBaseChange}
    The coherent sheaves $\HCal^i(Rf_*\Omega_{\ACal/S}^{\geq j})$ are finite locally free $\OCal_S$-modules and the morphisms $\HCal^i(Rf_*\Omega_{\ACal/S}^{\geq j})\longrightarrow\HCal^i(Rf_*\Omega_{\ACal/S}^{\bullet})
     $ are injective, thereby inducing isomorphisms $\HCal^i(Rf_*\Omega_{\ACal/S}^{\geq j})\iso Fil^j\HCal^i(Rf_*\Omega_{\ACal/S}^{\bullet})\subset \HCal^i(Rf_*\Omega_{\ACal/S}^{\bullet})$. Moreover, the formations of the sheaves $\HCal^i(Rf_*\Omega_{\ACal/S}^{\geq j})$ and the Hodge filtrations on $\HCal^i(Rf_*\Omega_{\ACal/S}^{\bullet})$ are compatible with (underived) base change.
\end{Lemma}
\begin{proof}
    
    We endow $\Omega^{\geq j}_{\ACal/S}$ with a filtration similarly to the naive filtration on $\Omega^{\bullet}_{\ACal/S}$. Then, we have an analogue of the Hodge to de Rham spectral sequence which computes cohomology sheaves of $Rf_*(\Omega^{\geq j}_{\ACal/S})$. (More precisely, we first choose a double complex injective resolution of $\Omega^{\geq j}_{\ACal/S}$ and use its totalization to obtain a complex of injective objects $I^\bullet$ quasi-isomorphic to $\Omega^{\geq j}_{\ACal/S}$. $I^\bullet$  remembers the filtered structure because totalizing one of the two naive filtrations on the double complex lets us resolve the filtered pieces of $\Omega^{\geq j}_{\ACal/S}$. This way, $I^\bullet$ defines a filtered complex of injectives. The filtered complex spectral sequence of $f_*(I^\bullet)$ with the induced filtration is what we needed.)
    The morphism $\Omega^{\geq j}_{\ACal/S}\longrightarrow\Omega^{\bullet}_{\ACal/S}$ respects the filtration, so that we can consider the induced morphism of spectral sequences. The maps on the $E_1^{r,s}$ terms are $R^sf_*(Gr^r(\Omega_{\ACal/S}^{\geq j}))\longrightarrow R^sf_*(Gr^r(\Omega^{\bullet}_{\ACal/S}))$ (which are either the identity or the inclusion of the zero sheaf). By Lemma \ref{BBMLemma}, both of these spectral sequences degenerate at the $E_1$-page, so this description also applies to the $E_{\infty}$-page.
    The nonzero terms of the former spectral sequence have the form $R^sf_*(\Omega^r_{\ACal/S})$ which are locally free by Lemma \ref{BBMLemma}. It follows that $\HCal^i(Rf_*\Omega^{\geq j}_{\ACal/S})$ can be written as a successive extension of locally free coherent modules. It then must be flat, hence locally free by induction.  

    For the injectivity of $\HCal^i(Rf_*\Omega_{\ACal/S}^{\geq j})\longrightarrow\HCal^i(Rf_*\Omega_{\ACal/S}^{\bullet})
     $, we look at the map between the spectral sequences mentioned above. As noted above, the map of the spectral sequences at the $E_\infty$-page are injective, so the various $Gr$'s of the map $\HCal^i(Rf_*\Omega_{\ACal/S}^{\geq j})\longrightarrow\HCal^i(Rf_*\Omega_{\ACal/S}^{\bullet})
     $ are injective. By descending induction on $s$ and four lemma, we see that $Fil^s\HCal^i(Rf_*\Omega_{\ACal/S}^{\geq j})\longrightarrow Fil^s\HCal^i(Rf_*\Omega_{\ACal/S}^{\bullet})
     $ is injective for any $s$.

    For the base change statement, consider base-changing along the map $g:S'\longrightarrow S$.
By the spectral sequence $E_2^{r,s}=L^rg^*\HCal^s(Rf_*\Omega^{\geq j}_{\ACal/S})\implies \HCal^{r+s}(Lg^*Rf_*\Omega^{\geq j}_{\ACal/S})$ and the local-freeness proved at the beginning, we see $g^*\HCal^i(Rf_*\Omega^{\geq j}_{\ACal/S})\iso \HCal^{i}(Lg^*Rf_*\Omega^{\geq j}_{\ACal/S})$.
 The formations of $Rf_*\Omega_{\ACal/S}^\bullet,Rf_*\Omega_{\ACal/S}^{\geq j}$ are compatible with (derived) base change by \cite[\href{https://stacks.math.columbia.edu/tag/0FM0}{Tag 0FM0}]{stacks-project} and its proof, so we are done.
\end{proof}
\begin{Lemma}[Field case]\label{HdgFilField}
    Let $A$ be an abelian variety over a field $k$ (which can be non-algebraically closed and of any characteristic). Then, the cup product map $\bigwedge^mH^1_{\dR}(A/k)\longrightarrow H^m_{\dR}(A/k)$ induces an isomorphism of filtered vector spaces.
\end{Lemma}
\begin{proof}
By Lemma \ref{BBMLemma} and the subsequent discussions, it suffices to compare the dimensions of the filtered pieces, or equivalently the dimensions of the graded pieces. Let $g=\dim A.$ Since $\dim H^1_{\dR}(A/k)=2g, \dim Fil^1H^1_{\dR}(A/k)=g,$ we see that $Gr^i(\bigwedge^mH^1_{\dR}(A/k))=\binom{g}{i}\binom{g}{m-i}$.
    By 2.5.2 \cite{BBM}, the Hodge to de Rham spectral sequence degenerates at the $E_1$ page, so 
    \[Gr^i(H^m_{\dR}(A/k))=\dim H^{m-i}(A,\Omega^i)=\binom{g}{i}\dim H^{m-i}(A,\OCal_A)=\binom{g}{i}\binom{g}{m-i},\] where the last equality used 
    \cite[\S13 Corollary 2]{Mumford}.
\end{proof}
\subsection{Grothendieck-Messing theory via Mukai structures}
Grothendieck-Messing theory describes lifts of abelian varieties via the data of Hodge filtrations.
In this subsection, specializing to the case of canonical lifts, we rewrite the theory in terms of the Mukai $F$-isocrystals. We assume throughout the subsection that $k$ is an algebraically closed field of characteristic $p>2$, and we will use the notations $W=W(k),K=\Frac W$. We briefly recall the notions of $F$-(iso)crystals.
\begin{Def}
\begin{enumerate}
    \item
    An \textbf{$F$-crystal} $(M,F)$ over $k$ consists of a finite free $W$-module $M$ and an injective $\sigma$-linear map $F:M\longrightarrow M$, where $\sigma:W\longrightarrow W$ is the Witt vector Frobenius operator.
    \item An \textbf{$F$-isocrystal} $(V,F)$ over $k$ consists of a finite dimensional $K$-vector space $V$ and an injective $\sigma$-linear map $F:V\longrightarrow V$.
    \item When $(M,F:M\longrightarrow M)$ is an $F$-crystal over $k$, for each non-positive integer $i\in\integ_{\leq 0}$ , its \textbf{$i$-th Tate twist} $M(i)$ is the $F$-crystal $(M,p^{-i}F:M\longrightarrow M)$.
    \item When $(V,F:V\longrightarrow V)$ is an $F$-isocrystal over $k$, for each integer $i\in\integ$ , its \textbf{$i$-th Tate twist} $V(i)$ is the $F$-isocrystal $(V,p^{-i}F:V\longrightarrow V)$. 
    \end{enumerate}
\end{Def}
\begin{Ex}
    If $X$ is a smooth proper $k$-variety, for each $i$, the crystalline cohomology modulo torsion $H_{\cris}^i(X/W)/(\rm torsion)$ with its Frobenius action is an $F$-crystal over $k$, and $H^i_{\cris}(X/K)$ is an $F$-isocrystal over $k$, where the injectivity condition of $F$-(iso)crystals follows from Poincar\'e duality.
\end{Ex}

We first note the following $F$-crystal analogue of Proposition \ref{HdgFilWedg} (which one can essentially deduce from this proposition by the compatibility of cup product and Berthelot-Ogus comparison isomorphism \cite[\S7]{Berthelot-Ogus-notes-crystalline-cohomology}):
\begin{Lemma}\label{FCrysWedg}
    Let $A$ be an abelian variety over $k$. For each non-negative integer $m$, the cup product induces an isomorphism
    \[
\cup_m:\bigwedge^mH_{\cris}^1(A/W)\longrightarrow H_{\cris}^m(A/W)
    \] of $F$-crystals.
\end{Lemma}
Specializing to the case of an ordinary abelian variety $A$ of dimension $g$ over $k$,
we recall from \cite{DelCris} that we have a unique direct sum decomposition
$H_{\cris}^1(A/W)=P_1\oplus P_0$, where these $W$-submodules are defined by $P_1\defeq H_{\cris}^1(A/W)^{F=p}\tens_{\integ_p}W\iso W(-1)^{\oplus g}$ and $P_0\defeq H_{\cris}^1(A/W)^{F=1}\tens_{\integ_p}W\iso W^{\oplus g}$. To see this, we can use ordinarity to produce the unique lift of the conjugate filtration (\cite[Proposition 1.3.2(iii)]{DelCris}) and then argue that it uniquely splits (\cite[Remarque 1.3.4]{DelCris}). Here, we note that the algebraic closedness of $k$ implies that any unit-root $F$-crystal is trivial (\cite[Proposition 1.2.2]{DelCris}, \cite[Proposition 4.1.1]{KatzPAdic}).
\begin{Def}[{Slope filtration}]
Let $(H,F)$ be an $F$-crystal with the property that $H=\bigoplus_{s\in\integ_{\geq0}} H^{F=p^s}\tens_{\integ_p}W$. Then $H$ has the \textbf{slope filtration} $Fil_{\slope}^*$ given by 
\[Fil^i_{\slope}H\defeq \bigoplus_{s\geq i} H^{F=p^s}\tens_{\integ_p}W\subset H.\]
By inverting $p$ in the slope filtration, we also obtain a \textbf{``rational'' slope filtration} on the induced $F$-isocrystal $(H[1/p],F[1/p])$.
\end{Def}
\begin{Ex}\label{ExHdgFil}
Let $A$ be an ordinary abelian variety over $k$.
The slope filtration $Fil_{\slope}^*$ on $H_{\cris}^1(A/W)$ is the descending filtration defined by \[
Fil^0_{\slope}H_{\cris}^1(A/W)=H_{\cris}^1(A/W),\]
\[Fil^1_{\slope}H_{\cris}^1(A/W)=P_1,\]
\[Fil^2_{\slope}H_{\cris}^1(A/W)=0.
\]
The mod $p$ reduction of the slope filtration coincides with the Hodge filtration of $H_{\dR}^1(A/k)$ in view of \cite[(1.3.1.2)]{DelCris} (this description of the Hodge filtration boils down to Mazur's theorem, as explained in \cite[1.3.1]{DelCris}).
\end{Ex}
Lemma \ref{FCrysWedg} immediately implies the following:
\begin{Prop}\label{SlpFilWedg}
Let $A$ be an ordinary abelian variety over $k$.
For a non-negative integer $m$, if we endow $H_{\cris}^m(A/W)$ with the slope filtration and $\bigwedge^mH_{\cris}^1(A/W)$ with the exterior power slope filtration, the isomorphism in Lemma \ref{FCrysWedg}
    defines an isomorphism of filtered $W$-modules:
    \[
\cup_m:\bigwedge^mH_{\cris}^1(A/W)\longrightarrow H_{\cris}^m(A/W).
    \] 
\end{Prop}
The classical Grothendieck-Messing theory describes the canonical lift as follows:
\begin{Prop}[{\cite[2.1.3]{DelCris}}]
    Let $A$ be an ordinary abelian variety over $k$, and let $\ACal$ be a lift of $A$ over $W$. 
    Then, $\ACal$ is the canonical lift if and only if the Hodge filtration on $H_{\cris}^1(A/W)$ induced via the Berthelot-Ogus comparison isomorphism $H_{\cris}^1(A/W)\iso H_{\dR}^1(\ACal/W)$ coincides with the slope filtration.
\end{Prop}
In the sequel, we will reformulate this theory in terms of Mukai $F$-isocrystals. 
\begin{Cor}\label{rationalGM}
 Let $A$ be an ordinary abelian variety over $k$, and let $\ACal$ be a lift of $A$ over $W$. 
 Then, $\ACal$ is the canonical lift if and only if the Hodge filtration on $H_{\cris}^1(A/K)$ induced via the isomorphism $H_{\cris}^1(A/K)\iso H_{\dR}^1(\ACal_{\eta}/K)$ coincides with the rational slope filtration.
\end{Cor}
\begin{proof}
From
Lemma \ref{BBMLemma}, $H_{\dR}^1(\ACal/W)$ is an extension of the finite free $W$-modules $H^1(\ACal,\OCal_{\ACal})$
and $H^0(\ACal,\Omega_{\ACal/W}^1)$.
This splits, so we deduce that the first step of the Hodge filtration on $H_{\dR}^1(\ACal/W)$ is a direct summand and hence saturated. The first step of the slope filtration in $H_{\cris}^1(A/W)$ is also a direct summand and saturated. Hence, the result follows from the proposition and the base change property of Hodge filtration (Lemma \ref{HdgFilBaseChange}).
\end{proof}
Next, we introduce the Mukai structure on cohomologies.
The main motivation for considering this structure comes from the fact that this is preserved under Fourier-Mukai equivalences (e.g.\ see the proof of Proposition \ref{CanFMPAV}).
Recall that when $(V,F^\bullet)$ is a filtered vector space, for any integer $i$, its $i$-th Tate twist $V(i)$ is the filtered vector space $(V,F_{(i)}^\bullet)$ with $F_{(i)}^jV=F^{i+j}V$.
\begin{Def}[Mukai structure on the odd de Rham and crystalline cohomologies, cf.\ \cite{LO}]
\leavevmode
Let $A$ be an abelian variety over $k$ as before. 
\begin{enumerate}
    \item 
We call the $F$-isocrystal $\bigoplus_{i=0}^{g-1}H_{\cris}^{2i+1}(A/K)(i)$ the \textbf{odd crystalline realization, or the odd Mukai $F$-isocrystal} of $A$.
\item Additionally assume that $A$ is ordinary for a moment. Then the odd crystalline realization $\bigoplus_{i=0}^{g-1}H_{\cris}^{2i+1}(A/K)(i)$ also carries the structure of a filtered vector space if we endow it with the rational slope filtration. Here, there is no ambiguity about the Tate twist $(i)$ (One interpretation shifts the filtration, and the other changes the $F$-isocrystal structure) because both interpretations yield the same filtration.
\item 
    Let $\ACal$ be a lift of $A$ over $W$. We call the filtered $K$-vector space $\bigoplus_{i=0}^{g-1}H_{\dR}^{2i+1}(\ACal_{\eta}/K)(i)$ the \textbf{odd de Rham realization, or the odd Mukai de Rham structure} of $\ACal_{\eta}$.
\end{enumerate}
\end{Def}
\begin{Rem}
    Both realizations carry the Poincar\'e pairing structures. 
    On one side, there is a perfect Poincar\'e pairing coming from the cup product on the crystalline cohomology $H_{\cris}^*(A/W)$ (\cite{Berth}; however the Tate twists used here requires the rational structure), and on the other side, we have the Poincar\'e pairing respecting the 
    Hodge filtration (\cite[\href{https://stacks.math.columbia.edu/tag/0FM7}{Tag 0FM7}]{stacks-project}).
    The Berthelot-Ogus comparison isomorphism is compatible with cup product, 
    so the pairing structure on these realizations are compatible. 
    The usual Poincar\'e pairing for odd de Rham cohomologies takes the form
    \[H_{\dR}^{2i+1}(\ACal_\eta/K)\tens H_{\dR}^{2(g-i-1)+1}(\ACal_\eta/K)\longrightarrow K(-g)
    \] for each $i$. By taking an alternating sum of these maps with Tate twists,
    we obtain
    \[\bigoplus_{i=0}^{g-1} \left(H_{\dR}^{2i+1}(\ACal_\eta/K)(i)\tens H_{\dR}^{2(g-i-1)+1}(\ACal_\eta/K)(g-i-1)\right)\longrightarrow K(-1).
    \] From this map, we obtain the Poincar\'e pairing on the odd Mukai de Rham structure:
\[
\left(\bigoplus_{i=0}^{g-1}H_{\dR}^{2i+1}(\ACal_{\eta}/K)(i)\right)
\tens \left(\bigoplus_{i=0}^{g-1}H_{\dR}^{2i+1}(\ACal_{\eta}/K)(i)\right)
\longrightarrow K(-1).
\]     
    The case of Mukai $F$-isocrystals is similar.
\end{Rem}
The key reformulation of the Grothendieck-Messing theory (for canonical lifts) in terms of the Mukai structure is the following:
\begin{Prop}\label{CanLiftViaMukai}
    Let $\ACal$ be a lift of $A$ over $W$. $\ACal$ is the canonical lift if and only if the Berthelot-Ogus comparison isomorphism for the (odd) crystalline and de Rham realiztions identify the slope filtration and the Hodge filtration.
\end{Prop}
\begin{proof}
Suppose that the two filtrations are identified.
    The Berthelot-Ogus comparison isomorphism respects the direct sum decomposition by construction, so restricting to $H^1$ lets us apply Corollary \ref{rationalGM}, and we see $\ACal$ is the canonical lift.

    Conversely, if $\ACal$ is the canonical lift, Corollary \ref{rationalGM} shows the filtrations on $H^1$ are identified. Proposition \ref{HdgFilWedg} and Proposition \ref{SlpFilWedg} extend this to the whole Mukai structures.
\end{proof}
\subsection{Proofs of the main results for abelian varieties}
We continue to assume that $k$ is an algebraically closed field of characteristic $p>2$. 
We start by recalling preliminary facts needed for our main results.
\begin{Prop}\label{ord D inv}
Let $A$ and $B$ be abelian varieties over an algebraically closed field of positive characteristic. In any of the following situations, if $A$ is ordinary (resp.\ supersingular), so is $B$.
\begin{enumerate}
    \item $B$ is isogenous to $A$.
    \item $B$ is derived equivalent to $A$.
    \item $B$ is a sub-abelian variety of $A$.
\end{enumerate}
\end{Prop}
\begin{proof}
We recall that the Hodge polygon of an abelian variety is determined by its dimension. Since the Newton polygon is isogeny-invariant, the case of an isogeny follows from \cite[III Th\'eor\`eme 3.2]{Ch}.

If $B$ is derived equivalent to $A$, there exists an isomorphism of abelian varieties $A\times \widehat{A}\iso B\times \widehat{B}$ (see Proposition \ref{RouPopaSch}). 
In particular, we have 
\[H^1_{\cris}(A/W)\oplus H^1_{\cris}(\widehat{A}/W)\iso H^1_{\cris}(B/W)\oplus H^1_{\cris}(\widehat{B}/W)\] 
as $F$-crystals.
Since an abelian variety is isogenous to its dual abelian variety, they have the same Newton polygons. Thus, the isomorphism of $F$-crystals above shows that $A$ and $B$ have the same Newton polygons. This treats the second case (see \cite{Honigs-Thesis} for a similar argument).

Finally, if $B$ is a sub-abelian variety of $A,$ then
 $B\times (A/B) $ and $A$ are isogenous.
Since the Newton polygon of $A$ is the same as the Newton polygon of $B\times (A/B)$, $B$ must be ordinary (resp.\ supersingular) in order for $A$ to be ordinary (resp.\ supersingular). 
\end{proof}
The key ingredient for Theorem \ref{MainLift} in the abelian variety case is the classification of point objects on abelian varieties. This classification is a variant of the following result in characteristic $0$:
\begin{Prop}[{\cite[Theorem 1.1]{DO}} over an algebraically closed field]\label{DO-point-obj-classification}
Let $L$ be an algebraically closed field of characteristic 0, and let $X$ be an abelian variety over $L$ of dimension $d$. Let $F\in\Dbcoh(X)$ be an object. Then, the following conditions are equivalent:
\begin{enumerate}
    \item 
    $
\Ext^{<0}(F,F)=0,\quad \Ext^0(F,F)=L,\quad\dim \Ext^1(F,F)\leq d.
$
\item There is a sub-abelian variety $i:Y\inj X$, a rational point $x\in X(L)$, a simple semi-homogeneous vector bundle $V$ on $Y$, and an integer $r$ such that $F\iso (t_x)_*i_*V[r]$, where $t_x:X\longrightarrow X$ is the translation map by $x$. 
\end{enumerate}
\end{Prop}
In our variant of this result, the condition $\dim \Ext^1(F,F)\leq d$ will be replaced by a condition on the dimension of a certain group scheme, which we will construct now.
Let $L$ be an algebraically closed field of any characteristic, and let $X$ be an abelian variety over $L$ of dimension $d$. 
Let $s\Dcal$ be the moduli stack of simple relatively perfect universally gluable complexes
for the morphism $X\longrightarrow \Spec L$. This is an algebraic stack quasi-separated and locally of finite type over $L$ (\cite{Lie}, \cite[\href{https://stacks.math.columbia.edu/tag/0DPX}{Tag 0DPX}]{stacks-project}),
and it's a $\multGpSch$-gerbe over the coarse space $sD$ which is an algebraic space quasi-separated and locally of finite type over $L$ (\cite[Corollaire 10.8]{CA}). 
Given a simple gluable complex $F\in\Dbcoh(X)$, it defines an $L$-point $[F]:\Spec L\longrightarrow s\Dcal$, or an $L$-point $[F]:\Spec L\longrightarrow sD$.
Similarly to \cite[5.8]{DO}, we consider the action map of $X\times_L\widehat{X}$ on $sD$ denoted as $\mu_F:X\times_L\widehat{X}\longrightarrow sD$ which acts on the $L$-points as $(x,L)\mapsto
(t_x^*F)\tens L$. We define the stabilizer $S_F$ by the cartesian square:
\[
\xymatrix{
S_F\ar[r]\ar[d]&\Spec L\ar[d]^{[F]}\\
X\times_L\widehat{X}\ar[r]^-{\mu_F}& sD
}
\]
By the following lemma, $S_F$ is a closed subgroup variety of $X\times_L\widehat{X}$:
\begin{Lemma}
Let $L$ be an algebraically closed field. 
    Let $M$ be an algebraic space quasi-separated and locally of finite type over $L$. Then, any section $s:\Spec L\longrightarrow M$ defines a closed immersion.
\end{Lemma}
\begin{proof}
We choose an \'etale cover $U\longrightarrow M$ by a scheme and prove that the base-changed map $s_U:U\times_M\Spec L\longrightarrow U$ is a closed immersion.
By further covering $U$ by Zariski opens, we assume $U$ is affine and of finite type over $L$. The morphism $s$ is quasi-compact, 
so $U\times_M\Spec L$ is a quasi-compact \'etale scheme over $L$. Then, $U\times_M\Spec L$ is a disjoint union of finitely many copies of $\Spec L$. Since $U$ is a finite type $L$-scheme, the map $U\times_M\Spec L\longrightarrow U$ must be a closed immersion.
\end{proof}
With this preparation, we can now state our variant of Proposition \ref{DO-point-obj-classification} in all characteristics:
\begin{Prop}[Variant of \cite{DO}]\label{DOmainResultVariant}
Let $L$ be an algebraically closed field, and let $X$ be an abelian variety of dimension $d$ over $L$.
Let $F\in\Dbcoh(X)$ be an object satisfying
\[
\Ext^{<0}(F,F)=0,\quad \Ext^0(F,F)=L,\quad\dim S_F\geq d.
\]
Then there is a sub-abelian variety $i:Y\inj X$, a rational point $x\in X(L)$, a simple semi-homogeneous vector bundle $V$ on $Y$, and an integer $r$ such that $F\iso (t_x)_*i_*V[r]$.
\end{Prop}
\begin{proof}
This proposition can be proved by a minor modification of the proof of one implication of Proposition \ref{DO-point-obj-classification} in \cite{DO}. As mentioned in \cite[1.3]{DO}, the characteristic 0 assumption in this proof can be dropped except for the section \cite[6.1]{DO}. In this step, de Jong and Olsson use the condition $\Ext^1(F,F)\leq d$ as well as the characteristic assumption to deduce that $\dim S_F\geq d$, and this is the only place where the assumption $\Ext^1(F,F)\leq d$ is used, so the rest of the proof goes through if we instead assume $\dim S_F\geq d$.
\end{proof}
\begin{Rem}
Let us briefly compare the conditions  
$\dim S_F\geq d$ and $\Ext^1(F,F)\leq d$ for an object $F\in \Dbcoh(X)$ such that $\Ext^{<0}(F,F)=0$ and $ \Ext^0(F,F)=L$. First, we notice that the former condition is stronger than the latter by the combination of Proposition \ref{DOmainResultVariant} and the  implication $(2)\implies(1)$ of Proposition \ref{DO-point-obj-classification} (which holds in all characteristics).  
As discussed in \cite[6.1]{DO}, the tangent space of $S_F$ at the origin is the kernel of a certain map $\kappa_F:HH^1(X)\to \Ext^1(F,F)$. Since the dimension of $HH^1(X)$ is $2d$, the condition $\Ext^1(F,F)\leq d$ implies that the dimension of the tangent space is at least $ d.$ In characteristic 0, the group variety $S_F$ is reduced and hence smooth, so this would imply $\dim S_F\geq d$, making the two assumptions equivalent.
\end{Rem}

We return to an abelian variety over $k$ and apply the above proposition:

\begin{Lemma}\label{liftPointLike}
   Let $A$ be an ordinary abelian variety over $k$ of dimension $d$. Any object $F\in\Dbcoh(A)$ satisfying \[
\Ext^{<0}(F,F)=0,\quad \Ext^0(F,F)=L,\quad\dim S_F\geq d.
\] is a coherent sheaf on $A$ (up to a shift) and lifts to a $W$-flat coherent sheaf $\widetilde{F}$ on the canonical lift $A_{\can}$ (up to a shift).
\end{Lemma}
\begin{proof}
By Proposition \ref{DOmainResultVariant}, we know that there are a sub-abelian variety $\iota:B\inj A$, a rational point $a\in A(k)$, an integer $m$ and a simple semi-homogeneous vector bundle $G$ on $B$ such that $F=(t_a)_*\iota_*G[m]$. 
    By \cite[Theorem 5.8]{Mukai}, there is some isogeny $\pi:C\surj B$ and a line bundle $L$ on $C$ such that $G\iso \pi_*L.$
By Lemma \ref{ord D inv}, $B$ and $C$ are ordinary. Hence, $A,B,C$ have canonical lifts $A_{\can},B_{\can},C_{\can}$, and the morphisms between them also lift canonically (\cite[Appendix]{MS}). Additionally, $L$ lifts to a line bundle $L_{\can}$ on $C_{\can}$, which is even unique if we require $L_{\can}^{\tens p}\iso \Frob^*L_{\can}$, where $\Frob:C_{\can}\longrightarrow C_{\can}$ is the Frobenius lift compatible with the Witt vector Frobenius on $W$ via the unit section $e_{C_{\can}}:\Spec W\longrightarrow C_{\can}$ (\cite[Appendix]{MS}). Observe that the canonical lift of $ t_a\circ\iota\circ \pi$ is finite by upper-semicontinuity of fiber dimensions.
We have the two Cartesian squares:
\[
\xymatrix{
C\ar[rr]^-{t_a\circ\iota\circ \pi}\ar[d]^\alpha&&A\ar[d]^\beta\ar[r]&\Spec k\ar[d]\\
C_{\can}\ar[rr]^-{(t_a\circ\iota\circ \pi)_{\can}}&&A_{\can}\ar[r]&\Spec W
}
\]
The outer square and the right square are tor-independent since the canonical lifts are flat over $W$. By \cite[Lemma 3.62]{Jiang}, it follows that the left square is tor-independent. We claim that $\widetilde{F}\defeq((t_a\circ\iota\circ \pi)_{\can})_*L_{\can}[m]$ lifts $F$. Indeed, by the tor-independent base change on the left square,
we have 
\begin{align*}
L\beta^*((t_a\circ\iota\circ \pi)_{\can})_*L[m]
&=L\beta^*R((t_a\circ\iota\circ \pi)_{\can})_*L[m]\\
&\iso R(t_a\circ\iota\circ \pi)_*L\alpha^*L_{\can}[m]\\
&=(t_a\circ\iota\circ \pi)_*L[m]\\
&=F.
\end{align*}
\end{proof}
In particular, this applies to kernels of Fourier-Mukai equivalences:
\begin{Prop}[Lifting of kernels]\label{CanLiftAVKer}
Let $A$ and $B$ ordinary abelian varieties over $k$. Let $E\in\DperfAbs(A\times_kB)$ be a kernel defining a Fourier-Mukai equivalence. Then, there exists a kernel $\ECal\in\DperfAbs(A_{\can}\times_W B_{\can})$ which lifts $E$ and accordingly defines a $W$-relative Fourier-Mukai equivalence between $A_{\can}$ and $B_{\can}$.
\end{Prop}
\begin{proof}
By Corollary \ref{LiftEquiv}, it suffices to check that $E$ satisfies the conditions of Lemma \ref{liftPointLike}.
Since $E$ defines an equivalence of categories $\Dbcoh(A)\iso \Dbcoh(B)$, for each $i$, we have
\[
\Ext_{A\times_kB}^i(E,E)
\iso \Ext_{A\times_kA}^i(\OCal_{\Delta_A},\OCal_{\Delta_A})
=HH^i(A).
\]
This vanishes for $i<0$ and has dimension 1 for $i=0$ (Note that these do not require the HKR isomorphisms to hold), so the first two conditions are satisfied.
For the last condition, we are looking at the stabilizer of $E$ in $(A\times_k B)\times_k(\widehat{A}\times_k\widehat{B})$.
\cite[Corollary 2.13]{Orlov} shows that a rational point
$(a,b,\alpha,\beta)\in(A\times_k B\times_k \widehat{A}\times_k \widehat{B})(k)$ satisfies 
$f_E(a,\alpha)=(b,\beta)$ if and only if $(a,b,\alpha,-\beta)$ is in the stabilizer of $E$, where $f_E:A\times_k\widehat{A}\iso B\times_k\widehat{B}$ is the isomorphism as in \cite[Theorem 2.10]{Orlov}. In other words, the inversion of $(S_F)_{\red}$ on the $\widehat{B}$-component coincides with the graph $\Gamma_{f_E}$ of the map $f_E$.
This has dimension $2g$, so we are done.
\end{proof}
\begin{Rem}
The proposition implies in particular that if $A$ and $B$ are derived equivalent ordinary abelian varieties over $k$, then $A_{\can,\overline{\eta}},B_{\can,\overline{\eta}}$ are also derived equivalent.

We can prove this weaker statement in the following different fashion. By \cite[Theorem 2.19]{Orlov}, we have an ``isometric'' isomorphism of abelian varieties $A\times \widehat{A}\iso B\times\widehat{B}$ (see \cite{Orlov} for the definition of being ``isometric.''
See also Proposition \ref{RouPopaSch}). Observe that the dual of the canonical lift is the canonical lift of the dual (see Example \ref{ExDual} below).
Then, the isomorphism $A\times \widehat{A}\iso B\times\widehat{B}$ lifts uniquely to an isomorphism of abelian schemes $A_{\can}\times\widehat{A_{\can}}\iso B_{\can}\times \widehat{B_{\can}}$ (which can be proved either by Serre-Tate theory or by deformation theory of schemes with Frobenius lifts as in \cite[Appendix]{MS}). This isomorphism is automatically ``isometric'' because group homomorphisms lift uniquely to the canonical lifts.
On the geometric generic fibre, we obtain an isometric isomorphism 
$A_{\can,\overline{\eta}}\times\widehat{A_{\can,\overline{\eta}}}\iso B_{\can,\overline{\eta}}\times \widehat{B_{\can,\overline{\eta}}}$. By \cite[Theorem 4.13]{Orlov} or \cite[\S15]{Polish}, we see $\Dbcoh(A_{\can,\overline{\eta}})\iso \Dbcoh(B_{\can,\overline{\eta}})$. 
\end{Rem}
\begin{Rem}\label{NonUniqueLift}
In contrast to the K3 surface case below (Proposition \ref{K3LiftKer}), we do not expect that we can uniquely lift kernels. For demonstration, let $A,B,E$ be as in Proposition \ref{CanLiftAVKer} and assume $p>\dim A$ for a moment. 
Then, as in \cite{LODef}, the set of deformations of $E$ to $(A_{\can}\times_WB_{\can})\times_WW_2(k)$ is a torsor under 
\[\Ext^1_{A\times_k B}(E,E\tens^L pW/p^2W)\iso \Ext^1_{A\times_k B}(E,E) \tens_k pW/p^2W\iso \Ext^1_{A\times_k B}(E,E).\]
Since $E$ defines an equivalence $\Dbcoh(A)\iso\Dbcoh(B)$, we have 
\[\Ext^1_{A\times_k B}(E,E)\iso \Ext^1_{A\times_k A}(\OCal_\Delta,\OCal_\Delta)=HH^1(A).\]
By the characteristic assumption, we have an HKR isomorphism \[HH^1(A)\iso H^1(A,\OCal_A)\oplus H^0(A,T_A)\]
which has dimension $2\dim A$ (\cite[Corollary 0.6]{Amnon}).
\end{Rem}
Next, we discuss an application to the $p$-adic variational Hodge conjecture. The conjecture predicts the following:
\begin{Conj}[$p$-adic variational Hodge conjecture (\cite{Emerton})]
Let $\XCal$ be a smooth proper $W(k)$-scheme, where $k$ is a perfect field of characteristic $p>0$. Suppose that $Z_0$ is a codimension $n$ cycle in $\XCal\times_{W(k)} k$ such that its cycle class \[cl(Z_0)\in H_{\cris}^{2n}(\XCal\times_{W(k)}k/W(k))\tens \Frac(W(k))\iso H_{\dR}^{2n}(\XCal/W(k))\tens\Frac(W(k))\] lies in the $n$-th piece of the Hodge filtration. Then, there exists a cycle $Z\subset \XCal$ which specializes to $Z_0$ (in the rational Chow group).
\end{Conj}
When $\XCal=A_{\can}$ for an ordinary abelian variety $A$, the condition in the conjecture regarding the Hodge filtration is vacuous. To see this, by the compatibility of Chern classes with cycle classes
(\cite[(16)]{GrothCh}),
it suffices to check that the crystalline Chern classes always land in the middle Hodge filtration under the identification $H_{\cris}^*(A/W)\iso H_{\dR}^*(A_{\can}/W)$. 
By the proof of Proposition \ref{CanLiftViaMukai}, the isomorphism identifies the Hodge filtration with the slope filtration, but the $i$-th crystalline Chern class always lands in the $i$-th filtered piece of the slope filtration (\cite[\S2.2]{LO}), so we have verified the condition. Lemma \ref{liftPointLike} yields the following special case of this conjecture (which in particular applies to kernels of Fourier-Mukai equivalences by the proof of Proposition \ref{CanLiftAVKer}):
\begin{Prop}\label{padicVarHdgAV}
Let $X$
be an ordinary abelian variety of dimension $d$ over $k$. Let $F\in\Dbcoh(X)$ be an object satisfying
\[
\Ext^{<0}(F,F)=0,\quad \Ext^0(F,F)=k,\quad\dim S_F\geq d.
\] We consider its Chern class $c_i(F)\in\CH^{i}(X)_\rat$ for each $i\geq0$. If we view $X$ as the special fiber of the $W(k)$-scheme $X_{\can}$, the $p$-adic variational Hodge conjecture holds for the class $c_i(F)$.
\end{Prop}
We move on to the next key observation, which will be used to describe the Fourier-Mukai partners of canonical lifts.
\begin{Prop}\label{CanFMPAV}
    Let $\XCal,\YCal$ be smooth projective $W$-schemes which are $W$-relatively Fourier-Mukai equivalent. If the special fiber of $\XCal$ is an ordinary abelian variety and $\XCal$ is its canonical lift, then the same is true for $\YCal$.
\end{Prop}
\begin{proof}
By Corollary \ref{FMPASch} and Lemma \ref{ord D inv}, 
$\YCal$ is an abelian scheme whose special fiber is ordinary. 
Let $\ECal$ be a kernel defining a relative Fourier-Mukai equivalence. 
As in \cite[\S2]{LO}, the $W$-relative Fourier-Mukai equivalence (or its restriction to the generic fiber) induces a cohomological Fourier-Mukai transform on the (odd) Mukai de Rham structures:
\[
\Phi^{\dR}_{\ECal_{\eta}}:\bigoplus_{i=0}^{g-1} H_{\dR}^{2i+1}(\XCal_{\eta}/K)(i)\longrightarrow \bigoplus_{i=0}^{g-1} H_{\dR}^{2i+1}(\YCal_{\eta}/K)(i)
\]
Note that this respects the filtered structure since Chern classes land in the middle piece of the Hodge filtration (see \S\ref{secChern}). We also have the inverse transform, so this is an isomorphism.
The restriction of the same relative Fourier-Mukai equivalence to the special fiber similarly defines an isomorphism of the (odd) Mukai $F$-isocrystals:
\[
\Phi^{\cris}_{\ECal_{s}}:\bigoplus_{i=0}^{g-1} H_{\cris}^{2i+1}(\XCal_{s}/K)(i)\longrightarrow \bigoplus_{i=0}^{g-1} H_{\cris}^{2i+1}(\YCal_{s}/K)(i)
\]
Moreover, these cohomological Fourier-Mukai transforms are compatible with the Berthelot-Ogus comparison isomorphisms by Corollary \ref{ChernCompat}:
\[
\xymatrix{
\bigoplus_{i=0}^{g-1} H_{\dR}^{2i+1}(\XCal_{\eta}/K)(i)\ar[r]_\iso^{\Phi^{\dR}_{\ECal_{\eta}}}\ar[d]^\iso &
\bigoplus_{i=0}^{g-1} H_{\dR}^{2i+1}(\YCal_{\eta}/K)(i)\ar[d]^\iso\\
\bigoplus_{i=0}^{g-1} H_{\cris}^{2i+1}(\XCal_{s}/K)(i)\ar[r]^{\Phi^{\cris}_{\ECal_{s}}}_\iso& \bigoplus_{i=0}^{g-1} H_{\cris}^{2i+1}(\YCal_{s}/K)(i)
}
\]
We are done by Proposition \ref{CanLiftViaMukai}.
\end{proof}
\begin{Ex}[Dual abelian variety and canonical lift]\label{ExDual}
    Let $A$ be an ordinary abelian variety over $k$. The Poincar\'e bundle on $A_{\can}\times\widehat{{A_{\can}}}$ defines a $W$-relative Fourier-Mukai equivalence between $A_{\can}$ and $\widehat{A_{\can}}$. It then follows that $\widehat{A_{\can}}$ is the canonical lift of its special fiber: $\widehat{A_{\can}}\iso (\widehat{A})_{\can}$.
\end{Ex}
\begin{Prop}[Fourier-Mukai partners of canonical lifts]\label{FMPofCanLiftAV}
    Let $X$ be an ordinary abelian variety over $k$. Then, the restriction to the special fiber defines a bijection:
    \[
    r:\FMP(X_{\can}/W)\longrightarrow \FMP(X/k)
    \]
    In particular, the set $\FMP(X_{\can}/W)$ is finite.
\end{Prop}
\begin{proof}
    Base change preserves Fourier-Mukai partners (Remark \ref{baseChangeFMP}), so the map is well-defined.

    By Lemma \ref{ord D inv}, Proposition \ref{FMPAV} and Corollary \ref{FMPASch}, the Fourier-Mukai partners under consideration are ordinary abelian varieties over $k$ and abelian schemes over $W$. Moreover by Proposition \ref{CanFMPAV}, such abelian schemes over $W$ arise as the canonical lifts of their special fibers.
    By Proposition \ref{CanLiftAVKer}, we have a section of $r$ given by
    \[s:\FMP(X/k)\longrightarrow \FMP(X_{\can}/W);\quad
    [Y]\mapsto [Y_{\can}],\] but this is surjective by what we just wrote.

    For the finiteness statement, it suffices to know that there are only finitely many Fourier-Mukai partners of $X$ over $k$ up to isomorphism, but this is proved in \cite[Corollary 2.20]{Orlov}. 
\end{proof}
\section{K3 surface case}\label{secK3}
In this section, we show that most results from \S\ref{secAV} can be analogously proved for K3 surfaces. 
Throughout the section, we assume $k$ is an algebraically closed field of characteristic $p>2.$
We first recall that the property of being a K3 surface is derived-invariant: 
\begin{Lemma}[Fourier-Mukai partners of K3 surfaces, {\cite[Proposition 3.9]{LO}}]\label{derInvK3}
    Let $X$ be a K3 surface over $k$. Then, any Fourier-Mukai partner of $X$ over $k$ is again a K3 surface.
\end{Lemma}
\subsection{Crystalline cohomology}   
We briefly recall the crystalline cohomology of K3 surfaces for readers' convenience. As in \cite[Proposition 1.1]{DelK3}, when $X$ is a K3 surface over $k$, $H^i_{\cris}(X/W)$ is a free $W$-module of rank $1,0,22,0,1$ for $i=0,1,2,3,4$ respectively. 
By \cite[II Corollaire 3.5, II Corollaire 3.15]{Ill}, 
the slopes of $H^0_{\cris}(X/W)$ (resp.\ $H^4_{\cris}(X/W)$) lie in $[0,1)$ (resp.\ $[2,3)$). However, they are free of rank 1, so we have isomorphisms of $F$-crystals 
\[H^0_{\cris}(X/W)\iso W,H^4_{\cris}(X/W)\iso W(-2).\]
We also recall that 
the Hodge numbers of $X$ are the same as those of a complex K3 surface (\cite[Proposition 1.1]{DelK3}), so that its ordinarity can be checked by looking at the Newton polygon (\cite[III Th\'eor\`eme 3.2]{Ch}).
As in \cite{LO}, we will look at the \textbf{(even) Mukai $F$-isocrystal}
of $X$, which is given by 
\[H_{\cris}^0(X/K)\oplus H_{\cris}^2(X/K)(1)\oplus H_{\cris}^4(X/W)(2)\iso H_{\cris}^2(X/K)(1)\oplus K^{\oplus 2},\]
where we are using a slightly different convention about the Tate twist from \cite{LO}. In addition to the $F$-isocrystal structure, it also comes with the Mukai pairing (see \cite{LO}).
As in the abelian variety case, the (even) Mukai $F$-isocrystals are preserved under derived equivalence, so we obtain:
\begin{Lemma}\label{D inv ord K3}
    If $X$ and $Y$ are derived equivalent $K3$ surfaces over $k$, they have the same Newton polygons. In particular, ordinarity and supersingularity of K3 surfaces are derived-invariant.
\end{Lemma}
 Now, we additionally assume that $X$ is ordinary.
    As in the abelian variety case, ordinarity implies that the $F$-crystal $H_{\cris}^2(X/W)$ decomposes as a direct sum of copies of $W$ with different Tate twists. Ordinarity 
    also lets us read off the Newton polygon from the Hodge numbers. Thus, we obtain a slope filtration by free $W$-submodules
    \[0=Fil_{\slope}^3\subset Fil_{\slope}^2\subset Fil^1_{\slope}\subset Fil_{\slope}^0=H_{\cris}^2(X/W(k)),\]
where the ranks of the filtered pieces increase as $0,1,21,22.$ The perfect Poincar\'e pairing on $H^2_{crys}(X/W(k))$ respects the $F$-crystal structure, so we see $Fil_{\slope}^1=(Fil_{\slope}^2)^\perp$. We also endow the (even-degree) Mukai $F$-isocrystal \[H_{\cris}^0(X/K)\oplus H_{\cris}^2(X/K)(1)\oplus H_{\cris}^4(X/W)(2)\]  with the slope filtration.
\subsection{De Rham cohomology}
Let $\XCal$ be a lift of $X$ over $W(k)$.
The argument of \cite[Proposition 2.2]{DelK3}
shows that if $\XCal$ is a lift of a K3 surface $X$ from $k$ to $W(k)$ with structure map $f:\XCal\longrightarrow\Spec W(k)$, then $R^if_*\Omega^j_{\XCal/W(k)}$ is finite free, the base change map $R^if_*\Omega_{\XCal/W(k)}^j\tens_{W(k)} k\longrightarrow H^i(X,\Omega_{X/k}^j)$ is an isomorphism, the cup product $H^2_{\dR}(\XCal/W(k))\times H^2_{\dR}(\XCal/W(k))\longrightarrow H^4_{\dR}(\XCal/W(k))$ defines a perfect duality, and the relative Hodge to de Rham spectral sequence of $\XCal/W(k)$ degenerates at the $E_1$ page
(the reference works with the universal deformation, but the same argument works for a particular lift over $W(k)$). 
Then, by an argument similar to the abelian variety case (Lemma \ref{HdgFilBaseChange}), 
we get an isomorphism $R^2f_*\Omega_{\XCal/W(k)}^{\geq i}\iso Fil^iR^2f_*\Omega_{\XCal/W(k)}^\bullet$ (\cite[2.3]{DelK3}). This way, we get a Hodge filtration by free $W(k)$-submodules 
\[0=Fil_{\Hdg}^3\subset Fil_{\Hdg}^2\subset Fil^1_{\Hdg}\subset Fil_{\Hdg}^0=R^2f_*\Omega_{\XCal/W(k)}^\bullet=H_{\dR}^2(\XCal/W(k)),\]
where the ranks of the filtered pieces increase as $0,1,21,22.$
Using the perfect duality, we have $Fil^1_{\Hdg}=(Fil_{\Hdg}^2)^\perp$. 
We also obtain from the above 
that 
\[0=Fil^1_{\Hdg}H^0_{\dR}(\XCal/W)\subset Fil^0_{\Hdg}H^0_{\dR}(\XCal/W)=H^0_{\dR}(\XCal/W).\]
\[
0=Fil^3_{\Hdg}H^4_{\dR}(\XCal/W)\subset Fil^2_{\Hdg}H^4_{\dR}(\XCal/W)=H^4_{\dR}(\XCal/W).
\]
Similarly to the ``crystalline realization,'' the \textbf{(even) Mukai de Rham realization} of $\XCal_{\eta}$ is $H^0_{\dR}(\XCal_{\eta}/K)\oplus H^2_{\dR}(\XCal_{\eta}/K)(1)\oplus H^4_{\dR}(\XCal_{\eta}/K)(2)$ which has a filtration coming from the Tate twists of Hodge filtrations on the direct summands. 
\subsection{Proofs of the main results for K3 surfaces}
Here is the classical description of the canonical lifts for K3 surfaces via Hodge filtration:
\begin{Prop}[Canonical lifts of K3 surfaces via Hodge filtration, {\cite[Proposition 2.1.9, Th\'eor\`eme 2.1.11]{DelCris}}]
    Let $\XCal$ be a lift of an ordinary K3 surface $X$ from $k$ to $W$. $\XCal$ is the canonical lift of $X$ if and only if the Hodge filtration and the slope filtration are identified via the comparison isomorphism $H_{\cris}^2(X/W)\iso H_{\dR}^2(\XCal/W)$.
    This is also equivalent to requiring that the two submodules $Fil_{\Hdg}^2H_{\dR}^2(\XCal/W)$ and $Fil^2_{\slope}H^2_{\cris}(X/W)$ are identified via the comparison isomorphism.
\end{Prop}
Using our knowledge of the Hodge/slope filtrations on $H^0$ and $H^4$ above, similarly to the abelian variety case, we can upgrade this to the following:
\begin{Prop}[Canonical lifts of K3 surfaces via Mukai structures]\label{CanLiftViaMukaiK3}
    Let $\XCal$ be a lift of an ordinary K3 surface $X$ from $k$ to $W$. $\XCal$ is the canonical lift of $X$ if and only if the Hodge filtration and the slope filtration are identified via the Berthelot-Ogus comparison isomorphism 
    between the (even) Mukai $F$-isocrystal and the (even) Mukai de Rham realization of $\XCal_{\eta}$.
\end{Prop}
\begin{Cor}\label{CanFMPK3}
    Let $\XCal,\YCal$ be smooth projective $W(k)$-schemes which are Fourier-Mukai equivalent relative to $W(k)$. If the special fiber of $\XCal$ is an ordinary K3 surface, and $\XCal$ is its canonical lift, the same holds for $\YCal$. 
\end{Cor}
\begin{proof}
This is proved as in Proposition \ref{CanFMPAV} using the even degree Mukai structures for K3 surfaces,  Lemma \ref{derInvK3}, Lemma \ref{D inv ord K3} and Proposition \ref{CanLiftViaMukaiK3}.
\end{proof}
The following is the K3 case of the Theorem \ref{MainLift} (and it implies the K3 case of Corollary \ref{MainLiftCor}). Subsequent results in this section are essentially corollaries of results in \cite{LO} (see \cite[Theorem 1.3, \S7]{LO}):
\begin{Prop}[Lifting of kernels]\label{K3LiftKer}
Let $X$ and $Y$ be ordinary K3 surfaces, and assume that a kernel $E\in\Dbcoh(X\times_k Y)$ defines a Fourier-Mukai equivalence.
Then, $E$ lifts to a perfect complex $\ECal\in\DperfAbs(X_{\can}\times_{W(k)}Y_{\can})$, where $X_{\can}$ and $Y_{\can}$ denote the canonical lifts of $X$ and $Y$ over $W(k)$.
Moreover, the lift of $E$ is unique.
\end{Prop}
\begin{proof}
The liftability of the kernel is an immediate consequence of \cite
[Theorem 7.1]{LO}. Indeed, by this theorem, it suffices to check that the composition
\[
\xymatrix{
\bigoplus_{i=0}^2 H_{\dR}^{2i}(X_{\can,\eta}/K)(i)\ar[r]^-{\iso}& 
\bigoplus_{i=0}^2 H_{\cris}^{2i}(X/K)(i)\ar[d]_{\iso}^{\Phi_{E}^{\cris}}\\
\bigoplus_{i=0}^2 H_{\dR}^{2i}(Y_{\can,\eta}/K)(i) &\bigoplus_{i=0}^2 H_{\cris}^{2i}(Y/K)(i)\ar[l]_-{\iso} 
}
\] sends $Fil_{\dR}^2$ to $Fil^2_{\dR}$. On the other hand, we know that the horizontal isomorphisms identify the Hodge filtrations with the slope filtrations (Proposition \ref{CanLiftViaMukaiK3}), and $\Phi_{E}^{\cris}$ respects the slope filtrations. This establishes the liftability.

Now, assume that $E$ defines a Fourier-Mukai equivalence. Corollary \ref{LiftEquiv} then shows that the lift $\ECal$ also defines a relative Fourier-Mukai equivalence. 
For the uniqueness of the lift, it suffices to establish the uniqueness formally due to the formal GAGA result for pseudo-coherent complexes (\cite[Theorem 4.10]{Lim}, \cite{Lie}). 

For each integer $n\geq0$, we let $Z_n\defeq (X_{\can}\times_W Y_{\can})\times_W{W_{n+1}(k)}$ and $\ECal_n\defeq \ECal|_{Z_n}$.
For the uniqueness, we inductively use that $\ECal_{n+1}$  is the unique lift of $\ECal_n$ (defined similarly) to $Z_{n+1}$.
As in \cite{LODef}, the set of such lifts is a torsor under 
\begin{align*}
\Ext^1_{Z_n}(\ECal_n,\ECal_n\tens^L_{\OCal_{Z_n}} p^nW/p^{n+1}W)
&\iso 
\Ext^1_{X\times_kY}(E,E\tens^L_{\OCal_{Z_0}} p^nW/p^{n+1}W)\\
&\iso \Ext^1_{Z_0}(E,E)\tens_{k} p^nW/p^{n+1}W.
\end{align*}
Since $E$ defines a derived equivalence $\Dbcoh(X)\iso\Dbcoh(Y)$, we have
\[
 \Ext^1_{Z_0}(E,E)\iso \Ext^1_{X\times X}(\OCal_{\Delta},\OCal_{\Delta})=HH^1(X).
\]
Due to our characteristic assumption, we can use the HKR isomorphism (\cite[Corollary 0.6]{Amnon}) to see $HH^1(X)\iso H^0(X,T_X)\oplus H^1(X,\OCal_X)=0$. (In fact, we could deduce $HH^1(X)=0$ in any characteristics, because the relevant terms in the $E_2$ page of the HKR spectral sequence are already 0. Also we note that the HKR spectral sequence degenerates for $p=2$ in this case because $\dim X=2.)$ 
This shows that the lifting is unique.
\end{proof}
\begin{Rem}
    A weaker version of the above can be obtained as follows as well. Let $X,Y$ be derived equivalent K3 surfaces over $k$. We know from \cite{LO} that $Y\iso M_v(X)$ can be written as a certain fine moduli space of stable sheaves on $X$. Then, the relative moduli space $\mathcal{M}_v(X_{\can})$ over $W(k)$ is fine and Fourier-Mukai equivalent to $X_{\can}$ via the universal stable sheaf (Corollary \ref{LiftEquiv}). Proposition \ref{CanFMPK3} implies that $Y_{\can}\iso \mathcal{M}_v(X_{\can})$, so $X_{\can}, Y_{\can}$ are relative Fourier-Mukai partners. This argument, however, does not show that arbitrary Fourier-Mukai kernels can be lifted.
\end{Rem}
\begin{Prop}[Fourier-Mukai partners of canonical lifts]\label{FMPofCanLiftK3}
    Let $X$ be an ordinary K3 surface over $k$. Then, the restriction to the special fiber defines a bijection:
    \[
    r:\FMP(X_{\can}/W)\longrightarrow \FMP(X/k).
    \]
    In particular, the set $\FMP(X_{\can}/W)$ is finite.
\end{Prop}
\begin{proof}
The proof of the bijectivity is similar to the abelian variety case, and it follows from Lemma \ref{derInvK3}, Lemma \ref{D inv ord K3} and  Corollary \ref{CanFMPK3}. For finiteness, see \cite[\S9]{LO}.
\end{proof}
\section{Sample application to a result in \cite{Honigs}}\label{secHonigs}
In this section, we will exhibit a sample application of Theorem \ref{MainLift}.
Namely, we will demonstrate that we can use this theorem to reprove the following special case of a result due to Honigs, Lombardi and  Tirabassi:
\begin{Prop}[Special case of {\cite[Theorem 1.1]{Honigs}}]\label{HonigsApp}
    Let $S$ be a hyperelliptic surface over an algebraically closed field $k$ of characteristic $p>3$, and let $A$ be its canonical cover. Additionally assume that $A$ is ordinary.
    Then, any Fourier-Mukai partner of $A$ is isomorphic to $A$ or $\widehat{A}$.
\end{Prop}
\begin{Rem}
    This type of result was originally proved by Sosna in \cite{Sosna} over $\C$ (\cite[Theorem 1.1]{Sosna}), and \cite[Theorem 1.1]{Honigs} proved it in characteristics $p>3$ without the ordinarity assumption.
\end{Rem}
By the same argument as \cite{Honigs}, due to the classification result of hyperelliptic surfaces, Proposition \ref{HonigsApp} will follow once we show the following proposition:
\begin{Prop}[Special case of {\cite[Theorem 4.4]{Honigs}}]\label{HonigsLemma}
  Let $A$ be an abelian surface over $k$, which is an algebraically closed field of characteristic $p>3$. Let $E,F$ be two elliptic curves over $k$. In the following three cases, any Fourier-Mukai partner of $A$ is isomorphic to $A$ or $\widehat{A}$:
\begin{enumerate}
    \item $A\iso E\times_k F$.
    \item $A$ is a degree 2 \'etale cyclic cover of $E\times_k F$.
    \item $A$ is a degree 3 \'etale cyclic cover of $E\times_k F$, and $F$ has an automorphism of order 3.
\end{enumerate}
\end{Prop}
We will prove this proposition by lifting to characteristic 0 and then using the original proof of \cite{Sosna}. The ingredient from the characteristic 0 setting is the following: 
\begin{Prop}[{\cite[\S4]{Sosna}}, see also {\cite[Theorem 4.1]{Honigs}}]\label{SosnaInput}
Let $X$ be an abelian surface over $\C$. Let $E,F$ be two elliptic curves over $\C$. In the following three cases, any Fourier-Mukai partner of $X$ is isomorphic to $X$ or $\widehat{X}$:
\begin{enumerate}
    \item $X\iso E\times_k F$.
    \item $X$ is a degree 2 \'etale cyclic cover of $E\times_k F$.
    \item $X$ is a degree 3 \'etale cyclic cover of $E\times_k F$, and $\rho(F)$ is either 2 or 4.
\end{enumerate}
\end{Prop}
\begin{proof}[Proof of Proposition \ref{HonigsLemma}]
Let $B$ be a Fourier-Mukai partner of $A$ (which we know is an ordinary abelian surface by Proposition \ref{FMPAV} and Proposition \ref{ord D inv}). By Theorem \ref{MainLift}, $A_{\can}$ and $B_{\can}$ are $W$-relative Fourier-Mukai equivalent, so choosing an embedding $W\inj\C$, we see that
$(A_{\can})_{\C}$ and $(B_{\can})_{\C}$ are also Fourier-Mukai equivalent.
If $(A_{\can})_{\C}$ satisfies one of the three hypotheses of Proposition \ref{SosnaInput}, 
the proposition says that $(B_{\can})_{\C}$ is isomorphic to $(A_{\can})_{\C}$ or its dual. As is done in \cite{LO}, by spreading out and using \cite[Lemma 6.5]{LO}, we see that $B$ must be isomorphic to $A$ or its dual. (\cite[Lemma 6.5]{LO} works with K3 surfaces, but the same argument goes through fro abelian varieties. Also see the proof of \cite[Proposition 4.8]{Honigs}.)
Hence, the rest of the proof is to show that $(A_{\can})_{\C}$ satisfies one of the three conditions in Proposition \ref{SosnaInput}
.

In any of the three cases in Proposition \ref{HonigsLemma}, we are given an \'etale cyclic cover $f:A\longrightarrow E\times_kF$. Up to choosing the identity section appropriately, we may assume that this is a homomorphism of group schemes. By Lemma \ref{ord D inv}, $E$ and $F$ are ordinary, so $f$ can be lifted uniquely to a homomorphism of group schemes as 
\[
f_{\can}:A_{\can}\longrightarrow E_{\can}\times_WF_{\can}.
\]
$f$ is finite \'etale, so the same applies to $f_{\can}$. 
$(f_{\can})_*\OCal_{A_{\can}}$ is a finite free $\OCal_{E_{\can}\times_WF_{\can}}$-module, and by connectedness we see $f_{\can}$ has the same degree as $f$ globally. 
$\Ker(f_{\can})$ is an \'etale group scheme over $W$ whose special fiber is $(\integ/d\integ)_k$ 
where $d=\deg f$, so that $\Ker(f_{\can})=(\integ/d\integ)_W$. 
We then see that the geometric generic fiber $(f_{\can})_{\C}:(A_{\can})_{\C}\longrightarrow (E_{\can})_\C\times (F_{\can})_\C$ defines a cyclic covering of the same degree. At this point, we can invoke Proposition \ref{SosnaInput} in the first two cases. 

We now focus on the case where $f$ is a degree 3 covering and $F$ has an automorphism of order 3. Let $\sigma$ be this automorphism. It remains to show that $\rho((A_{\can})_\C)$ is 2 or 4. 
As before, $\sigma$ lifts uniquely to a group automorphism $\sigma_{\can}:F_{\can}\longrightarrow F_{\can}$.
We know $\sigma^3=\id_F$, so $(\sigma_{\can})^3=\id_{F_{\can}}$. 
The locus inside $F_{\can}$ where $\id_{F_{\can}}$ and $\sigma_{\can}$ coincide is closed by the separatedness of the map $F_{\can}\longrightarrow \Spec W$, so $\sigma_{\can}$ and $\id_{F_{\can}}$ must differ on the generic fiber, hence on the geometric generic fiber. It follows that $(F_{\can})_{\C}$ has an automorphism of order 3 and is a CM elliptic curve. 
Since isogenies preserve the Picard numbers, 
we have
$\rho((A_{\can})_\C)=\rho((E_{\can})_{\C}\times_\C (F_{\can})_{\C})$. 
If $(E_{\can})_{\C}$ and $(F_{\can})_{\C}$ are not isogenous, this number is 2 by \cite[Corollary 2.3]{PicNum}. 
Otherwise, by the isogeny-invariance of Picard number again, $\rho((E_{\can})_{\C}\times_\C (F_{\can})_{\C})=\rho((F_{\can})_{\C}\times_\C (F_{\can})_{\C})=4$  
(\cite[Corollary 2.6]{PicNum}).
\end{proof}

\appendix
\section{Crystalline and de Rham Chern classes}\label{secChern}
Let $k$ be a perfect field of characteristic $p$, and
let $\XCal$ be a smooth projective $W$-scheme. Let $X$ be its reduction modulo $p$.
Given a line bundle $L$ on $X$, we can define its first crystalline Chern class $c_{1,\cris}(L)\in H^2_{\cris}(X/W)$ using the connecting map of the following short exact sequence of crystalline sheaves
\[
0\longrightarrow 1+\JCal_{X/W}\longrightarrow\OCal_{X/W}^\times\longrightarrow i_*\OCal_{X}^\times\longrightarrow0
\]
and the map $\log:\JCal_{X/W}+1\longrightarrow\JCal_{X/W}\subset \OCal_{X/W}$ (see \cite{BerIll}, \cite[\S3]{BerOg} for details).

As in \cite{BerIll}, this can be (uniquely) extended to higher Chern classes. Namely, given a locally free $\OCal_{X}$-module $E$, we can associate its crystalline Chern classes $c_{i,\cris}(E)\in H^{2i}_{\cris}(X/W)$ for various indices $i$ (the original reference works with Chern classes valued in $H^{2i}((X/W)_{\cris},\JCal_{X/W}^{[i]})$, but we are working with their targets along the maps 
$H^{2i}((X/W)_{\cris},\JCal_{X/W}^{[i]})\longrightarrow H^{2i}((X/W)_{\cris},\OCal_{X/W})$). These classes satisfy the usual compatibility with pullback and additivity, and when $E=L$, the total Chern class is $1+c_{1,\cris}(L)$ whose construction is as above. 

On the other hand, the map $d\log:\OCal_{\XCal}^\times\longrightarrow\Omega^1_{\XCal/W};f\mapsto\frac{df}{f}$ lets us construct the first de Rham Chern class $c_{1,\dR}(-)\in H_{\dR}^2(\XCal/W)$ of line bundles on $\XCal$ 
(\cite[\S3]{BerOg}, \cite[\href{https://stacks.math.columbia.edu/tag/0FLE}{Tag 0FLE}]{stacks-project}). By construction, these Chern classes land in $Fil^1H^2_{\dR}(\XCal/W)$.
By splitting principle, 
the de Rham Chern class also extends to higher Chern classes (for smooth projective $W$-schemes) (see \cite[\href{https://stacks.math.columbia.edu/tag/0FW8}{Tag 0FW8}]{stacks-project}). 
By this construction and \cite[\href{https://stacks.math.columbia.edu/tag/0FM7}{Tag 0FM7}]{stacks-project},  the $i$-th de Rham Chern class lands in $Fil^iH_{\dR}^{2i}(\XCal/W)$. 
These two theories are compatible in the following sense:
\begin{Prop}[{\cite[Proposition (3.4)]{BerOg}}]
    If $\LCal$ is a line bundle on $\XCal$, with its restriction $L$ on $X$, then $c_{1,\cris}(L)$ maps to $c_{1,\dR}(\LCal)$ via the Berthelot-Ogus isomorphism
    $
    H_{\cris}^2(X/W)\iso H_{\dR}^2(\XCal/W)
    .$
\end{Prop}
Given a vector bundle $\ECal$ on $\XCal$, the projective bundle construction used in the framework of splitting principle can be performed for $\XCal/W$ (which realizes the version for $X/k$ on the special fiber), the comparison result above extends to: 
\begin{Prop}
    For a locally free $\OCal_{\XCal}$-module $\ECal$, let $E$ be its restriction to $X$. 
    Then the total Chern classes $c_{\totale,\cris}(E)$ and $c_{\totale,\dR}(\ECal)$ map to each other via the comparison isomorphism
    $H_{\cris}^*(X/W)\iso H_{\dR}^*(\XCal/W)$.
\end{Prop}
The smooth projective $W$-scheme $\XCal$ has the resolution property (\cite[\href{https://stacks.math.columbia.edu/tag/0FDD}{Tag 0FDD}]{stacks-project}), so this result can be stated in terms of $K_0(\DperfAbs(\XCal))$ (\cite[\href{https://stacks.math.columbia.edu/tag/0FDJ}{Tag 0FDJ}]{stacks-project}):
\begin{Cor}\label{ChernCompat}
    Let $\XCal$ be a smooth projective $W$-scheme. Then, we have a commutative diagram
    \[
    \xymatrix{
    K_0(\DperfAbs(\XCal))\ar[r]\ar[d]_{c_{\totale,\dR}}&K_0(\DperfAbs(X))\ar[d]_{c_{\totale,\cris}}\\
    H^*_{\dR}(\XCal/W) \ar[r]^-{\iso }&H^*_{\cris}(X/W)
    }
    \]
    where the top horizontal map is the restriction to the special fiber, and the bottom horizontal map is the comparison isomorphism.
\end{Cor}

\bibliographystyle{alpha}
\bibliography{bib}

\end{document}